\documentclass{amsart}
\usepackage{amsmath}
\usepackage{paralist}
\usepackage{amsfonts}
\usepackage{amssymb}
\usepackage{amsthm}
\usepackage{amscd}
\usepackage{float}
\usepackage{tikz}
\usepackage{graphicx}
\usepackage[colorlinks=true]{hyperref}
\hypersetup{urlcolor=blue, citecolor=red}
\usepackage{hyperref}

\newtheorem{theorem}{Theorem}[section]
\newtheorem{corollary}[theorem]{Corollary}

\newtheorem{lemma}[theorem]{Lemma}
\newtheorem{proposition}[theorem]{Proposition}

\theoremstyle{definition}
\newtheorem{definition}[theorem]{Definition}
\newtheorem{remark}[theorem]{Remark}

\newcommand{\R}{\mathbb{R}}
\newcommand{\Z}{\mathbb{Z}}

\newcommand{\C}{\mathbb{C}}

\makeatletter
\@namedef{subjclassname@2020}{%
  \textup{2020} Mathematics Subject Classification}
\makeatother

\begin{document}

\title[Resolvent estimates for Maxwell's equations]{Resolvent estimates for time-harmonic Maxwell's equations in the partially anisotropic case}

\author{Robert Schippa}
\address{Fakult\"at f\"ur Mathematik, Karlsruher Institut f\"ur Technologie,
Englerstrasse 2, 76131 Karlsruhe, Germany}
\email{robert.schippa@kit.edu}

\begin{abstract}
We prove resolvent estimates in $L^p$-spaces for time-harmonic Max\-well's equa\-tions in two spatial dimensions and in three dimensions in the partially an\-iso\-tro\-pic case. In the two-dimensional case the estimates are sharp up to endpoints. We consider anisotropic permittivity and permeability, which are both taken to be time-independent and spatially homogeneous. For the proof we diagonalize time-harmonic Maxwell's equations to equations involving Half-Laplacians. We apply these estimates to infer a Limiting Absorption Principle in intersections of $L^p$-spaces and to localize eigenvalues for perturbations by potentials.
\end{abstract}


\maketitle

\section{Introduction and Main Results}

Maxwell's equations describe electromagnetic waves and consequently the propagation of light. We refer to the physics' literature for further query (cf. \cite{LandauLifschitz1990,FeynmanLeightonSands1964}).
Time-dependent Maxwell's equations in media in three spatial dimensions relate
\emph{electric and magnetic field} $ (\mathcal{E},\mathcal{B}):\R \times \R^3 \to  \C^3 \times \C^3$ with
 \emph{displacement and magnetizing fields} $(\mathcal{D},\mathcal{H}):\R \times \R^3 \to \C^3 \times \C^3$, the \emph{electric and magnetic current} $(\mathcal{J}_e,\mathcal{J}_m): \R \times \R^3 \to \C^3 \times \C^3 $, and \emph{electric and magnetic charges} $(\rho_e,\rho_m): \R \times \R^3 \to \C \times \C$:
\begin{equation}
\label{eq:Maxwell3D}
\left\{ \begin{array}{cl}
\partial_t \mathcal{D} &= \nabla \times \mathcal{H} + \mathcal{J}_e, \qquad \nabla \cdot \mathcal{D} = \rho_e, \quad \nabla \cdot \mathcal{B} = \rho_m, \\
\partial_t \mathcal{B} &= - \nabla \times \mathcal{E} + \mathcal{J}_m.
\end{array} \right.
\end{equation}
In physical contexts, fields, currents and charges are real-valued, and the magnetic charge and current vanish. We consider possibly non-vanishing magnetic charge and current to highlight symmetry between the electric and magnetic field. Moreover, $\mathcal{J}_e$ and $\mathcal{J}_m$ are typically taken with opposite signs.

In the following we consider the time-harmonic, monochromatic ansatz
\begin{equation}
\label{eq:TimeHarmonicAnsatz}
\begin{split}
\mathcal{D}(t,x) &= e^{i \omega t} D(x), \quad \mathcal{H}(t,x) = e^{i \omega t} H(x), \\
 \mathcal{J}_e(t,x) &= e^{i \omega t} J_{e}(x), \quad \mathcal{J}_m(t,x) = e^{i \omega t} J_{m}(x)
 \end{split}
\end{equation}
with $\omega \in \R$. We supplement \eqref{eq:Maxwell3D} with the material laws 
\begin{equation}
\label{eq:MaterialLaws}
\mathcal{D}(t,x) = \varepsilon \mathcal{E}(t,x), \quad \mathcal{B}(t,x) = \mu \mathcal{H}(t,x),
\end{equation}
where $\varepsilon = \text{diag}(\varepsilon_1,\varepsilon_2,\varepsilon_3) \in \R^{3 \times 3}, \; \varepsilon_i, \; \mu \in \R_{> 0}$. Requiring $\varepsilon$ and $\mu$ to be symmetric and positive definite is a physically natural assumption. 
The fully anisotropic case
\begin{equation*}
\varepsilon = \text{diag}(\varepsilon_1,\varepsilon_2,\varepsilon_3), \quad \mu = \text{diag}(\mu_1,\mu_2,\mu_3)  \text{ with } \frac{\varepsilon_1}{\mu_1} \neq \frac{\varepsilon_2}{\mu_2} \neq \frac{\varepsilon_3}{\mu_3} \neq \frac{\varepsilon_1}{\mu_1}
\end{equation*}
is analyzed in joint work with R. Mandel \cite{MandelSchippa2021}, where we argue in detail how the analysis reduces in the general case to scalar $\mu$ (see also \cite[p.~63]{Liess1991}).
 Material laws with scalar $\mu$ are frequently used in optics (cf. \cite[Section~2]{MoloneyNewell1990}). Then \eqref{eq:Maxwell3D} becomes under \eqref{eq:TimeHarmonicAnsatz} and \eqref{eq:MaterialLaws} to relate $E$ with $D$ and $H$ with $B$:
\begin{equation}
\label{eq:Maxwell3DConcise}
P(\omega,D) 
\begin{pmatrix}
D \\ B
\end{pmatrix}
 = \begin{pmatrix}
 J_e \\ J_m
 \end{pmatrix} , \quad P(\omega,D) = 
\begin{pmatrix}
i \omega & - \mu^{-1} \nabla \times \\
\nabla \times (\varepsilon^{-1} \cdot) & i \omega
\end{pmatrix}
.
\end{equation}
\eqref{eq:TimeHarmonicAnsatz} can be explained by considering \eqref{eq:Maxwell3D} under Fourier transforms in time: Letting
\begin{equation*}
\mathcal{D}(t,x) = \frac{1}{2 \pi} \int_{\R} e^{i \omega t} D(\omega,x) d\omega, \quad \mathcal{H}(t,x) = \frac{1}{2 \pi} \int_{\R} e^{i \omega t} H(\omega,x) d\omega, \ldots,
\end{equation*}
we find a solution to \eqref{eq:Maxwell3D} provided that $D(\omega,\cdot)$,... solve \eqref{eq:Maxwell3DConcise}. We focus on solenoidal currents, but shall also consider the effect of non-vanishing divergence. We deduce from the continuity equation for electric charges $\partial_t \rho_e(t,x) - \nabla \cdot \mathcal{J}_e(t,x) = 0$ the following relation between $J_e(\omega,\cdot)$ and the time-dependent charges:
\begin{equation*}
\nabla \cdot J_e(\omega,x) = i \omega \int_{\R} e^{-i \omega t} \rho_e(t,x) dt.
\end{equation*}
Since $\omega$ will be fixed in the following analysis of the time-harmonic equation, we let
\begin{equation}
\label{eq:3DCharges}
\rho_e(x) = \nabla \cdot J_e(x) \text{ and } \rho_m(x) = \nabla \cdot J_m(x).
\end{equation}

\medskip

We consider Maxwell's equations in two spatial dimensions and the partially anisotropic case in three dimensions. The time-dependent form of Maxwell's equations in two dimensions corresponds to electric and magnetic fields and currents of the form
\begin{align*}
\mathcal{E}_i(t,x) &= \mathcal{E}_i(t,x_1,x_2), \quad i=1,2; \quad \mathcal{E}_3 = 0; \\
\mathcal{B}_i &= 0, \quad i=1,2; \quad \mathcal{B}_3(t,x) = \mathcal{B}_3(t,x_1,x_2); \\
\mathcal{J}_{ei}(t,x) &= \mathcal{J}_{ei}(t,x_1,x_2), \quad i=1,2; \quad \mathcal{J}_{e3} = 0; \\
\mathcal{J}_{mi}(t,x) &= 0, \quad i=1,2; \quad \mathcal{J}_{m3}(t,x) = \mathcal{J}_{m3}(t,x_1,x_2).
\end{align*}
\eqref{eq:Maxwell3D} simplifies to (cf. \cite{2dMaxwellProposal}):
\begin{equation}
\label{eq:Maxwell2D}
\left\{ \begin{array}{cl}
\partial_t \mathcal{D} &= \nabla_{\perp} \mathcal{H} + \mathcal{J}_e, \quad  \nabla \cdot \mathcal{D}= \rho_e, \\
\partial_t \mathcal{B} &= - \nabla \times \mathcal{E} + \mathcal{J}_m,
\end{array} \right.
\end{equation}
where $\mathcal{D},\mathcal{E},\mathcal{J}_e:\R \times \R^2 \to \C^2$, $\mathcal{B},\mathcal{H},\mathcal{J}_m: \R \times \R^2 \to \C$, $\nabla_\perp = (\partial_2,-\partial_1)^t$, and we assume \eqref{eq:MaterialLaws} with $\mu > 0$, and $(\varepsilon^{ij})_{i,j} \in \R^{2 \times 2}$ denoting a symmetric,  positive definite matrix. We can rewrite \eqref{eq:Maxwell2D} under \eqref{eq:TimeHarmonicAnsatz} and \eqref{eq:MaterialLaws} as
\begin{equation}
\label{eq:Maxwell2dConcise}
P(\omega,D) \begin{pmatrix}
D \\ B
\end{pmatrix}
= \begin{pmatrix}
J_e \\ J_m
\end{pmatrix}, \quad
P(\omega,D) =
\begin{pmatrix}
i \omega & 0 & - \mu^{-1} \partial_2 \\
0 & i \omega & \mu^{-1} \partial_1 \\
 \partial_1 \varepsilon_{21} - \partial_2 \varepsilon_{11} &  \partial_1 \varepsilon_{22} - \partial_2 \varepsilon_{12} & i \omega
\end{pmatrix},
\end{equation}
denoting with $\varepsilon_{ij}$ the components of the inverse of $\varepsilon$. In two dimensions, we let 
\begin{equation}
\label{eq:Charges2D}
\rho_e = \partial_1 J_e + \partial_2 J_e \text{ and } \rho_m = 0.
\end{equation}

\vspace*{0.5cm}

In the following let $d \in \{2,3\}$, $m(2) = 3$, $m(3) = 6$, and
\begin{align*}
L_0^p(\R^2) &= \{ (f_1,f_2,f_3) \in L^p(\R^2)^3 \, : \, \partial_1 f_1 + \partial_2 f_2 = 0 \text{ in } \mathcal{S}'(\R^2) \}, \\
L_0^p(\R^3) &= \{ (f_1,\ldots,f_6) \in L^p(\R^3)^6 \, : \, \nabla \cdot (f_1,f_2,f_3) = \nabla \cdot( f_4,f_5,f_6) = 0 \text{ in } \mathcal{S}'(\R^3) \}.
\end{align*}

In this paper we are concerned with the resolvent estimates
\begin{equation}
\label{eq:ResolventEstimates}
\| (D,B) \|_{L_0^q(\R^d)} =  \| P(\omega, D)^{-1} (J_{e},J_{m}) \|_{L_0^q(\R^d)} \lesssim \kappa_{p,q}(\omega) \| (J_{e},J_{m}) \|_{L_0^p(\R^d)}.
\end{equation}

However, as will be clear from perceiving $P(\omega,D)$ as a Fourier multiplier,\\ $P(\omega,D)^{-1}$ cannot even be understood in the distributional sense for $\omega \in \R$. The remedy will be to consider $\omega \in \C \backslash \R$ and prove estimates independent of the distance to the real axis. Then we can consider limits $\Im (\omega) \downarrow 0$ and $\Im (\omega) \uparrow 0$. This is presently referred to as Limiting Absorption Principle (LAP) in the $L^p$-$L^q$-topology. Moreover, the analysis yields explicit formulae for the resulting limits. It appears that this is the first contribution to resolvent estimates for the Maxwell operator in anisotropic media in the $L^p$-$L^q$-topology.

\medskip

Recently, Cossetti--Mandel analyzed the isotropic\footnote{In the isotropic case we identify $\varepsilon = \lambda 1_{3 \times 3}$ with $\lambda \in \R_{>0}$ and do likewise for $\mu$.}, possibly spatially inhomogeneous case $\varepsilon, \mu \in W^{1,\infty}(\R^3; \R_{>0})$ in \cite{CossettiMandel2021}. In the isotropic case, iterating \eqref{eq:Maxwell3D} and using the divergence conditions yields Helmholtz-like equations for $D$ and $H$. This approach was carried out in \cite{CossettiMandel2021}. In the anisotropic case this strategy becomes less straight-forward. Instead we choose to diagonalize the Fourier multiplier to get into the position to use resolvent estimates for the fractional Laplacian. Kwon--Lee--Seo \cite{KwonLeeSeo2021} previously used a diagonalization to prove resolvent estimates for the Lam\'e operator. However, there are degenerate components in the diagonalization of time-harmonic Maxwell's operators, which do not occur for the Lam\'e operator. We use the divergence condition to ameliorate the contribution of the degeneracies. In case the currents have non-vanishing divergence, we can quantify this contribution with the charges.

\medskip

We digress for a moment to elaborate on $L^p$-$L^q$-estimates for the fractional Laplacian and applications. Let $s \in (0,d)$. For $\omega \in \C \backslash [0,\infty)$ we consider the resolvents as Fourier multiplier:
\begin{equation}
\label{eq:FourierMultiplier}
((-\Delta)^{s/2} - \omega)^{-1}f = \frac{1}{(2 \pi)^d} \int_{\R^d} \frac{\hat{f}(\xi)}{ \|\xi \|^s - \omega } e^{i x. \xi} d\xi
\end{equation}
for $f: \R^d \to \C$ in some suitable a priori class, e.g., $f \in \mathcal{S}(\R^d)$. In the present context, resolvent estimates for the Half-Laplacian $\|((-\Delta)^{\frac{1}{2}}-\omega)^{-1} \|_{p \to q}$ are most important. There is a huge body of literature on resolvent estimates for the Laplacian $(-\Delta - \omega)^{-1}:L^p(\R^d) \to L^q(\R^d)$. This is due to versatile applications to uniform Sobolev estimates and unique continuation (cf. \cite{KenigRuizSogge1987}), the localization of eigenvalues for Schr\"odinger operators with complex potential (cf. \cite{Cuenin2017,Frank2011,Frank2018}), or LAPs in $L^p$-spaces (cf. \cite{Gutierrez2004}). Kenig--Ruiz--Sogge \cite{KenigRuizSogge1987} showed that uniform resolvent estimates in $\omega \in \C \backslash [0,\infty)$ for $d \geq 3$ hold if and only if
\begin{equation}
\label{eq:UniformBoundedness}
\frac{1}{p} - \frac{1}{q} = \frac{2}{d} \text{ and } \frac{2d}{d+3} < p < \frac{2d}{d+1}.
\end{equation}
By homogeneity and scaling, we find
\begin{equation}
\label{eq:ScalingResolvent}
\| (-\Delta - \omega)^{-1} \|_{p \to q} = |\omega|^{-1+\frac{d}{2} \big( \frac{1}{p} - \frac{1}{q} \big)} \| \big( - \Delta - \frac{\omega}{|\omega|} \big)^{-1} \|_{p \to q} \quad \forall \omega \in \mathbb{C} \backslash [0, \infty).
\end{equation}
Thus, it suffices to consider $|\omega| = 1$ to discuss boundedness. Kwon--Lee \cite{KwonLee2020} showed the currently widest range of resolvent estimates for the fractional Laplacian outside the uniform boundedness range (see \cite{HuangYaoZheng2018} for a previous contribution). To state the range of admissible $L^p$-$L^q$-estimates, we shall use notations from \cite{KwonLee2020}.
Let $I^2 = \{(x,y) \in \R^2 \, | \, 0 \leq x,y \leq 1 \}$, and let $(x,y)^\prime = (1-x,1-y)$ for $(x,y) \in I^2$. For $\mathcal{R} \subseteq I^2$ we set $\mathcal{R}^\prime = \{ (x,y)^\prime \, | \, (x,y) \in \mathcal{R} \}$.\\
The resolvent of the fractional Laplacian $((-\Delta)^{\frac{s}{2}} - z)^{-1}$ is bounded for fixed $z \in \C \backslash [0,\infty)$ if and only if $(1/p,1/q) \in \mathcal{R}_0^{\frac{s}{2}}$ with
\begin{equation*}
\mathcal{R}_0^{\frac{s}{2}} = \mathcal{R}_0^{\frac{s}{2}}(d) = \{(x,y) \in I^2 \, | \, 0 \leq x-y \leq \frac{s}{d} \} \backslash \{(1,\frac{d-s}{d}), (\frac{s}{d},0) \};
\end{equation*}
see, e.g., \cite[Proposition~6.1]{KwonLee2020}. Guti\'{e}rrez showed in \cite{Gutierrez2004} that uniform estimates for $\omega \in \{ z \in \C \, : \, |z| = 1, \, z \neq 1 \}$ hold if and only if $(1/p,1/q)$ lies in the set
\begin{equation}
\label{eq:UniformBoundednessII}
\mathcal{R}_1 = \mathcal{R}_1(d) = \{ (x,y) \in \mathcal{R}^1_0(d) \, : \, \frac{2}{d+1} \leq x-y \leq \frac{2}{d}, \, x > \frac{d+1}{2d}, \, y < \frac{d-1}{2d} \}.
\end{equation}
Failure outside this range was known before (cf. \cite{KenigRuizSogge1987,Boerjeson1986}) due to the connection to Bochner-Riesz operators with negative index. Clearly, there are more estimates available outside $\mathcal{R}_1$ if one allows for dependence on $\omega$, e.g.,
\begin{equation*}
\| (-\Delta - \omega)^{-1} \|_{L^2 \to L^2} \sim \text{dist}(\omega,[0,\infty))^{-1}.
\end{equation*}
Kwon--Lee \cite{KwonLee2020} analyzed estimates outside the uniform boundedness range in detail and covered a wide range. Estimates with dependence on $\omega$ can be used to localize eigenvalues for Schr\"o\-din\-ger operators with complex potentials (cf. \cite{Cuenin2017}), which is done for Maxwell operators in Section \ref{section:Localization}.

\medskip

Diagonalizing the symbol of \eqref{eq:Maxwell3DConcise} to operators involving the Half-Laplacian works in the \textit{partially anisotropic case}, i.e.,
\begin{equation}
\label{eq:PartiallyAnisotropicCondition}
\# \{\varepsilon_1,\varepsilon_2,\varepsilon_3\} \leq 2.
\end{equation}
This includes the isotropic case $\varepsilon_1 = \varepsilon_2 = \varepsilon_3$, for which the results of Cossetti--Mandel \cite{CossettiMandel2021} are recovered for constant coefficients, albeit via a different approach.
It turns out that in the fully anisotropic case
\begin{equation*}
\varepsilon_1 \neq \varepsilon_2 \neq \varepsilon_3 \neq \varepsilon_1,
\end{equation*}
diagonalizing the multiplier introduces singularities, and this case has to be treated differently (cf. \cite{MandelSchippa2021}). The estimates proved in \cite{MandelSchippa2021} for the fully anisotropic case are strictly weaker than in the partially anisotropic case. We connect resolvent bounds for the Maxwell operator with resolvent estimates for the Half-Laplacian:
\begin{theorem}
\label{thm:ResolventEstimateMaxwell}
Let $1 < p,q < \infty$, $d \in \{2,3\}$, and $\omega \in \mathbb{C} \backslash \R$. Let $\varepsilon \in \R^{d\times d}$ denote a symmetric positive definite matrix, and let $P(\omega,D)$ as in \eqref{eq:Maxwell2dConcise} for $d=2$, and as in \eqref{eq:Maxwell3DConcise} for $d=3$. For $d=3$, we assume that $\varepsilon=\text{diag}(\varepsilon_1,\varepsilon_2,\varepsilon_3)$ and satisfies \eqref{eq:PartiallyAnisotropicCondition}.\\
Then, $P(\omega,D)^{-1}:L^p_0(\R^d) \to L^q_0(\R^d)$ is bounded if and only if $(1/p,1/q) \in \mathcal{R}_0^{\frac{1}{2}}(d)$, and we find the estimate
\begin{equation}
\label{eq:ResolventEquivalence}
\| P(\omega,D)^{-1} \|_{L^p_0 \to L^q_0} \sim \| ((-\Delta)^{\frac{1}{2}} - \omega)^{-1} \|_{L^p \to L^q} + \| ((-\Delta)^{\frac{1}{2}} + \omega)^{-1} \|_{L^p \to L^q}
\end{equation}
to hold. \\
If $1 \leq p \leq \infty $ and $1 < q < \infty$, then we find the estimate
\begin{equation}
\label{eq:ResolventEstimateAbove}
\begin{split}
&\quad \| P(\omega,D)^{-1} (J_e,J_m) \|_{L^q}  \\
&\lesssim (\| ((-\Delta)^{\frac{1}{2}} - \omega)^{-1} \|_{L^p \to L^q} + \| ((-\Delta)^{\frac{1}{2}} + \omega)^{-1} \|_{L^p \to L^q}) \|(J_e,J_m)\|_{L^p} \\
&\quad + \| (-\Delta)^{-\frac{1}{2}} \rho_e \|_{L^q} + \| (-\Delta)^{-\frac{1}{2}} \rho_m \|_{L^q}
\end{split}
\end{equation}
to hold with $\rho_e$ and $\rho_m$ defined as in \eqref{eq:Charges2D} for $d=2$ and \eqref{eq:3DCharges} for $d=3$. If $\rho_e = \rho_m = 0$, $1<p<\infty$, and $q \in \{1,\infty\}$, then \eqref{eq:ResolventEstimateAbove} also holds.
\end{theorem}
We cannot allow for $p\in \{1, \infty\}$ or $q \in \{1,\infty\}$ in the proof of 
\begin{equation*}
\| P(\omega,D)^{-1} \|_{L^p_0 \to L^q_0} \gtrsim \| ((-\Delta)^{\frac{1}{2}} - \omega)^{-1} \|_{L^p \to L^q} + \| ((-\Delta)^{\frac{1}{2}} + \omega)^{-1} \|_{L^p \to L^q}
\end{equation*}
as multiplier bounds for Riesz transforms are involved. It is well-known that the Riesz transforms are bounded on $L^p(\R^d)$, $1<p<\infty$, but neither on $L^1$ nor on $L^\infty$. In the proof of  \eqref{eq:ResolventEstimateAbove} for $\rho_e = \rho_m = 0$, which covers the reverse estimate of the above display, we can overcome this possibly technical issue by arranging the Riesz transforms acting on a reflexive $L^p$-space. Hence, we can allow for either $p \in \{1,\infty\}$ or $q \in \{ 1, \infty\}$. For the sake of simplicity, in Corollary \ref{cor:ResolventBounds} we only consider $1<p,q<\infty$ although \eqref{eq:ResolventEstimateAbove} partially extends to $p \in \{1, \infty\}$ or $q \in \{1, \infty\}$.

\medskip

Coming back to resolvent estimates for the Half-Laplacian, for $d \in \{2,3\}$ and $(1/p,1/q) \in I^2$, define
\begin{equation*}
\gamma_{p,q} = \gamma_{p,q}(d) = \max \{ 0, 1 - \frac{d+1}{2} \big( \frac{1}{p} - \frac{1}{q} \big), \frac{d+1}{2} - \frac{d}{p}, \frac{d}{q} - \frac{d-1}{2} \}.
\end{equation*}
Set
\begin{align*}
\kappa_{p,q}^{(\frac{1}{2})}(\omega) &= |\omega|^{-1 + d \big( \frac{1}{p} - \frac{1}{q} \big) + \gamma_{p,q}} \text{dist}(\omega,[0,\infty))^{-\gamma_{p,q}}, \\
\kappa_{p,q}(\omega) &= |\omega|^{-1 + d \big( \frac{1}{p} - \frac{1}{q} \big) + \gamma_{p,q}} \text{dist}(\omega,\R)^{-\gamma_{p,q}}.
\end{align*}
 Kwon--Lee \cite[Conjecture~3,~p.~1462]{KwonLee2020} conjectured for $(1/p,1/q) \in \mathcal{R}_0^{1/2}(d)$
\begin{equation}
\label{eq:SharpResolventEstimate}
\kappa_{p,q}^{(\frac{1}{2})}(\omega) \sim_{p,q,d} \| ((-\Delta)^{1/2}-\omega)^{-1} \|_{p \to q}.
\end{equation}
They verified the conjecture for $d=2$ and for $d = 3$ in the restricted range in $\tilde{\mathcal{R}}_0^{1/2}(3)$ \cite[Theorem~6.2,~p.~1462]{KwonLee2020}. We refer to \cite{KwonLee2020} for the precise description. For notational convenience, let $\tilde{\mathcal{R}}_0^{1/2}(2) = \mathcal{R}_0^{1/2}(2)$. By invoking the results from \cite{KwonLee2020}, we find the following:
\begin{corollary}
\label{cor:ResolventBounds}
Let $1 < p,q < \infty$, $d \in \{2,3\}$, and $\omega \in \mathbb{C} \backslash \R$. Let $\varepsilon \in \R^{d\times d}$ and $P(\omega,D)$ be as in Theorem \ref{thm:ResolventEstimateMaxwell}. Then we find the following:

\noindent 1. If $d=2$, then
\begin{equation}
\label{eq:ResolventEstimateMaxwell}
\| P(\omega,D)^{-1} \|_{L_0^p(\R^d) \to L^q_0(\R^d)} \sim \kappa_{p,q}(\omega)
\end{equation}
is true for $(1/p,1/q) \in \mathcal{R}_0^{\frac{1}{2}}(2)$.

\medskip

\noindent 2. If $d=3$ with $\varepsilon$ satisfying \eqref{eq:PartiallyAnisotropicCondition}, then \eqref{eq:ResolventEstimateMaxwell} is true for $(1/p,1/q) \in \tilde{\mathcal{R}}_0^{\frac{1}{2}}(3)$.
\end{corollary}

\medskip

Turning to LAPs, we work with the following notions:
\begin{definition}
Let $d \in \{2,3\}$, $1 \leq p,q \leq \infty$, $\omega \in \R \backslash 0$, and $0 < \delta < 1/2$. We say that a global $L^p_0$-$L^q_0$-LAP holds if $P(\omega \pm i \delta,D)^{-1}: L^p_0(\R^d) \to L^q_0(\R^d)$ are bounded uniformly in $\delta > 0$, and there are operators
$P_{\pm}(\omega) : L_0^p(\R^d) \to L_0^q(\R^d)$ such that
\begin{equation}
\label{eq:GlobalLAP}
P(\omega \pm i \delta,D)^{-1} f \to P_{\pm}(\omega) f \text{ as } \delta \to 0 \text{ in } (\mathcal{S}'(\R^d))^{m(d)}.
\end{equation}

We say that a local $L^p_0$-$L^q_0$-LAP holds if for any $\beta \in C^\infty_c(\R^d )$, $P(\omega \pm i \delta,D)^{-1} \beta(D): L^p_0(\R^d) \to L^q_0(\R^d)$ are bounded uniformly in $\delta > 0$, and there are operators $P^{loc}_{\pm}(\omega): L_0^p(\R^d) \to L_0^q(\R^d)$ such that
\begin{equation}
\label{eq:LocalLAP}
P(\omega \pm i \delta,D)^{-1} \beta(D) f \to P_{\pm}^{loc}(\omega) f \text{ in } \mathcal{S}'(\R^d)^{m(d)}.
\end{equation}
\end{definition}
\begin{remark}
By the explicit formulae for $P(\omega,D)^{-1}$ for $\omega \in \C \backslash \R$ we can also handle currents with non-vanishing divergence as in Theorem \ref{thm:ResolventEstimateMaxwell}. We omitted this discussion for the sake of brevity.
\end{remark}

We observe that $\gamma_{p,q} > 0$ for $p$ and $q$ as in Corollary \ref{cor:ResolventBounds}:
\begin{corollary}
\label{cor:GlobalLAP}
Let $d \in \{2,3\}$. For $1<p,q<\infty$, $(1/p,1/q) \in \tilde{\mathcal{R}}_0^{\frac{1}{2}}(d)$, there is no global $L^p_0$-$L^q_0$-LAP for \eqref{eq:Maxwell2dConcise} or \eqref{eq:Maxwell3DConcise}.
\end{corollary}

We show a local $L^p_0$-$L^q_0$-LAP for the Maxwell operator in Proposition \ref{prop:LocalLAP}. Roughly speaking, for low frequencies the resolvent estimates are equivalent to resolvent estimates for the Laplacian, and uniform estimates $L^{p_1} \to L^q$ are possible for $(1/p_1,1/q) \in \mathcal{P}(d)$ (see Section \ref{section:LAP}). For the high frequencies, away from the singular set, the multiplier is smooth, but provides merely the smoothing of the Half-Laplacian. We use different $L^{p_2} \to L^q$-estimates for this region. This gives  $L^{p_1} \cap L^{p_2} \to L^q$-estimates, which are uniform in $\omega$ in a compact set away from the origin, and an LAP in the same spaces. The necessity of considering currents in intersections of $L^p$-spaces is shown in Corollary \ref{cor:GlobalLAP}. Below for $s \geq 0$ and $1 < q < \infty$, $W^{s,q}(\R^d)$ denotes the $L^q$-based Sobolev space:
\begin{equation*}
W^{s,q}(\R^d) = \{ f \in L^q(\R^d) : (1-\Delta)^{s/2} f \in L^q \} \text{ and } \| f \|_{W^{s,q}} := \| (1-\Delta)^{s/2} f \|_{L^q}.
\end{equation*}
\begin{theorem}[LAP for Time-Harmonic Maxwell's equations]
\label{thm:LocalLAP}
Let $1 \leq p_1,p_2,q \leq \infty$, and let $d \in \{2,3\}$. If $(1/p_1,1/q) \in \mathcal{P}(d)$, $(1/p_2,1/q) \in \mathcal{R}_0^{\frac{1}{2}}(d)$, then $P(\omega,D)^{-1}: L_0^{p_1}(\R^d) \cap L_0^{p_2}(\R^d) \to L_0^q(\R^d)$ is bounded uniformly for $\omega \in \C \backslash \R$ in a compact set away from the origin. Furthermore, for $\omega \in \R \backslash 0$ there are limiting operators $P_{\pm}(\omega): L_0^{p_1}(\R^d) \cap L_0^{p_2}(\R^d) \to L_0^{q}(\R^d)$ with
\begin{equation*}
P(\omega \pm i \delta, D)^{-1} (J_e,J_m) \to P_{\pm}(\omega) (J_e,J_m) \text{ in } (\mathcal{S}'(\R^d))^{m(d)} \text{ as } \delta \downarrow 0
\end{equation*}
 such that $(D,B) = P_{\pm}(\omega) (J_e,J_m)$ satisfy
\begin{equation}
\label{eq:LimitingOperators}
P(\omega,D) (D,B) = (J_e,J_m) \text{ in } (\mathcal{S}'(\R^d))^{m(d)}.
\end{equation}
Additionally, if $q<\infty$, and $s \in [1,\infty)$, then
\begin{equation}
\label{eq:WeakSolutions}
\| (D,B) \|_{(W^{s,q}(\R^d))^{m(d)}} \lesssim \| (J_e,J_m) \|_{(W^{s-1,q}(\R^d))^{m(d)} \cap L_0^{p_1}(\R^d)}.
\end{equation}
\end{theorem}

Previously, Picard--Weck--Witsch \cite{PicardWeckWitsch2001} showed an LAP in weighted $L^2$-spaces (cf. \cite{Agmon1975}). Since the results in \cite{PicardWeckWitsch2001} are proved via Fredholm's Alternative, the frequencies $\omega \in \R \backslash 0$ are assumed not to belong to a discrete set of eigenvalues. In \cite{PicardWeckWitsch2001} $\varepsilon$ and $\mu$ are assumed to be positive-definite and isotropic, but allowed to depend on $x$ as in \cite{CossettiMandel2021}. Pauly \cite{Pauly2006} proved similar results as Picard--Weck--Witsch \cite{PicardWeckWitsch2001} in weighted $L^2$-spaces in the anisotropic case; see also \cite{PaulyThesis,BenArtziNemirovsky1998}. Much earlier, Eidus \cite{Eidus1985} already proved non-existence of eigenvalues of the Maxwell operator provided that $\varepsilon$ and $\mu$ are sufficiently smooth short-range perturbations of the identity and satisfy a repulsivity condition. Recently, D'Ancona--Schnaubelt \cite{DAnconaSchnaubelt2021} proved global-in-time Strichartz estimates from resolvent estimates in weighted $L^2$-spaces. 

\smallskip

It appears that in the present work the role of the Half-Laplacian is explicitly identified for the analysis of the Maxwell operator the first time. We note that in \cite{2dQuasilinearMaxwell,3dQuasilinearMaxwell}, in joint work with R. Schnaubelt, we apply a similar diagonalization to show Strichartz estimates for time-dependent Maxwell's equations with rough coefficients. In these works, due to variable permittivity and permeability, the diagonalization is carried out with pseudo-differential operators, and the present role of the Half-Laplacian is played by the Half-Wave operator. Provided that suitable estimates for the Half-Laplacian with variable coefficients were at disposal, of which the author is not aware, it seems possible that the present approach extends to variable permittivity and permeability as well.
\vspace*{0.5cm}

\textit{Outline of the paper.} In Section \ref{section:Reduction} we diagonalize time-harmonic Maxwell's equations in Fourier space to reduce the resolvent estimates to estimates for the Half-Laplacian. We also give examples for lower resolvent bounds in terms of the Half-Laplacian. In Section \ref{section:LAP} we argue how an LAP fails in $L^p$-spaces, but can be salvaged in intersections of $L^p$-spaces. In Section \ref{section:Localization} we show how the $\omega$-dependent resolvent estimates lead to localization of eigenvalues in the presence of potentials. We postpone technical computations to the Appendix, where we also give explicit solution formulae.

\section{Reduction to resolvent estimates for the Half-Laplacian}
\label{section:Reduction}
 Let $\omega \in \C \backslash \R$. We diagonalize $P(\omega,D)$ defined in \eqref{eq:Maxwell2dConcise} or in \eqref{eq:Maxwell3DConcise} in the partially anisotropic case. We shall see that the transformation matrices are essentially Riesz transforms. This allows to bound the resolvents with estimates for the Half-Laplacian. We will make repeated use of the Mikhlin--H\"ormander multiplier theorem (cf. {\cite[Theorem~6.2.7,~p.~446]{Grafakos2014}}):
\begin{theorem}[Mikhlin--H\"ormander]
\label{thm:MultiplierTheorem}
Let $1<p<\infty$ and $m: \R^n \backslash 0 \to \C$ be a bounded function that satisfies
\begin{equation}
\label{eq:DerivativeBounds}
|\partial^\alpha m(\xi)| \leq D_\alpha |\xi|^{-|\alpha|} \qquad (\xi \in \R^n \backslash 0)
\end{equation}
for $|\alpha| \leq \lfloor \frac{n}{2} \rfloor + 1$. Then, $\mathfrak{m}_p: L^p(\R^n) \to L^p(\R^n)$ given by $f \mapsto (m \hat{f}) \check{\;}$ defines a bounded mapping with
\begin{equation}
\label{eq:LpBoundMultiplier}
\| \mathfrak{m}_p \|_{L^p \to L^p} \leq C_n \max(p,(p-1)^{-1}) (A + \| m \|_{L^\infty}),
\end{equation}
where 
\begin{equation*}
A= \max(D_\alpha, \; |\alpha| \leq \lfloor \frac{n}{2} \rfloor + 1).
\end{equation*}
\end{theorem} 
 
As pointed out in \cite{Grafakos2014}, $m \in C^k(\R^n \backslash 0)$, $k \geq \lfloor \frac{n}{2} \rfloor + 1$ is an $L^p$-multiplier for $1<p<\infty$, if it is zero-homogeneous, i.e., there is $\tau \in \R$ such that for any $\lambda > 0$ and $\xi \neq 0$, we have
\begin{equation}
\label{eq:ZeroHomogeneous}
m(\lambda \xi) = \lambda^{i \tau} m(\xi).
\end{equation}

Differentiating the above display with respect to $\xi$, we obtain for $\lambda > 0$
\begin{equation*}
\lambda^{|\alpha|} (\partial_\xi^\alpha m)(\lambda \xi) = \lambda^{i \tau} \partial_\xi^\alpha m(\xi)
\end{equation*}
and \eqref{eq:DerivativeBounds} is satisfied with $D_\alpha = \sup_{|\theta| = 1} |\partial^\alpha m(\theta)|$.

\subsection{Proof of Theorem \ref{thm:ResolventEstimateMaxwell} for $d=2$} 
\label{subsection:2d}
  Let $u = (D_1,D_2,B)$. We denote $(\varepsilon^{-1})_{ij} = (\varepsilon_{ij})_{i,j}$. To reduce to estimates for the Half-Laplacian, we diagonalize the symbol associated with the operator defined in \eqref{eq:Maxwell2dConcise}. We write $\xi = (\xi_1,\xi_2) \in \R^2$:
 \begin{equation}
 \label{eq:MaxwellMultiplier}
(P(\omega,D) u) \widehat (\xi) = p(\omega,\xi) \hat{u}(\xi) = i
\begin{pmatrix}
\omega & 0 & -\xi_2 \mu^{-1} \\
0 & \omega & \xi_1 \mu^{-1} \\
\xi_1 \varepsilon_{12} - \xi_2 \varepsilon_{11} &  \xi_1 \varepsilon_{22} - \xi_2 \varepsilon_{12} & \omega
\end{pmatrix}
\hat{u}(\xi)
.
 \end{equation}
Let $\| \xi \|^2_{\varepsilon^\prime} = \langle \xi, \mu^{-1} \det(\varepsilon)^{-1} \varepsilon \xi \rangle $, $\xi^\prime = \xi / \| \xi \|_{\varepsilon^\prime}$, and define
\begin{equation}
\label{eq:FractionalLaplacianResolvent}
e_{\pm}(\omega,D): L^p(\R^2) \to L^q(\R^2), \quad (e_{\pm} f) \widehat (\xi) = \frac{1}{\omega \pm  \| \xi \|_{\varepsilon^\prime}} \hat{f}(\xi).
\end{equation}
We have the following lemma on diagonalization:
\begin{lemma}
\label{lem:Diagonalization2d}
For almost all $\xi \in \R^2$ there is a matrix $m(\xi) \in \C^{3 \times 3}$ such that 
\begin{equation*}
p(\omega,\xi) = m(\xi) d(\omega,\xi) m^{-1}(\xi)
\end{equation*}
with
\begin{equation}
\label{eq:Eigenvalues}
d(\omega,\xi) = i \text{diag}(\omega,\omega - \| \xi \|_{\varepsilon^\prime}, \omega + \| \xi \|_{\varepsilon^\prime}).
\end{equation}
Furthermore, the operators $m_{ij}(D)$ and $m^{-1}_{ij}(D)$ are $L^p$-bounded for $1<p<\infty$.
\end{lemma}
\begin{proof}
It is straight-forward to check that the eigenvalues are as in \eqref{eq:Eigenvalues} with the eigenvectors at hand. We align the corresponding eigenvectors as columns to
 \begin{equation}
 \label{eq:Eigenvectors}
m(\xi) = 
\begin{pmatrix}
 \varepsilon_{22} \xi_1^\prime -\varepsilon_{12} \xi_2^\prime  & -\xi_2^\prime \mu^{-1} & \xi_2^\prime \mu^{-1} \\
 \varepsilon_{11} \xi_2^\prime -\varepsilon_{12} \xi_1^\prime & \xi_1^\prime \mu^{-1} & -\xi_1^\prime \mu^{-1} \\
0 & -1 & -1
\end{pmatrix} 
 \end{equation}
 and note that $\det m(\xi) = - 1$ for $\xi \neq 0$. For the inverse matrix we compute
\begin{equation}
\label{eq:InverseEigenvectors}
m^{-1}(\xi) =
\begin{pmatrix}
 \mu^{-1} \xi_1^\prime & \mu^{-1} \xi_2^\prime & 0 \\
\frac{ \xi_1^\prime \varepsilon_{21} - \xi_2^\prime \epsilon_{11} }{2} & \frac{\varepsilon_{22} \xi_1^\prime - \varepsilon_{21} \xi_2^\prime }{2} & -\frac{1}{2} \\
\frac{\xi_2^\prime \varepsilon_{11} - \xi_1^\prime \varepsilon_{12}}{2} & \frac{ \xi_2^\prime \varepsilon_{12} - \xi_1^\prime \varepsilon_{22}}{2} & - \frac{1}{2}
\end{pmatrix}
.
\end{equation}
$L^p$-boundedness is immediate from Theorem \ref{thm:MultiplierTheorem} because the components of $m$ and $m^{-1}$ are zero-homogeneous and smooth away from the origin.
\end{proof}
In Proposition \ref{prop:Explicit2d} we compute $p^{-1}(\omega,\xi)$ via this diagonalization. The diagonalization allows us to separate
\begin{equation}
\label{eq:Decomposition2d}
p^{-1}(\omega,\xi) = M^2(A,B) + M^2_c
\end{equation}
with $M^2_c v = 0$ for $\xi_1 v_1 + \xi_2 v_2 = 0$ and
\begin{equation}
\label{eq:Constants2d}
A = \frac{1}{i(\omega - \| \xi \|_{\varepsilon'})}, \quad B = \frac{1}{i(\omega + \| \xi \|_{\varepsilon'})}.
\end{equation}
We can finish the proof of Theorem \ref{thm:ResolventEstimateMaxwell} for $d=2$:
\begin{proof}[Proof~of~Theorem~\ref{thm:ResolventEstimateMaxwell},~$d=2$]
We begin with the lower bound in \eqref{eq:ResolventEquivalence}. For $u$ with $\partial_1 u_1 +\partial_2 u_2 = 0$, we have
\begin{equation*}
p^{-1}(\omega,\xi) \hat{u}(\xi) = M^2(A,B) \hat{u}(\xi).
\end{equation*}
The entries of $M^2(A,B)$ are linear combinations of $e_{\pm}(\omega,\xi)$ and $\xi_i'$. The operators 
\begin{equation*}
(\mathcal{R}^{\varepsilon'}_i f) \widehat (\xi) = \xi_i' \hat{f}(\xi)
\end{equation*}
 are $L^p$-bounded for $1<p<\infty$ with a constant only depending on $p,\varepsilon,\mu$ as the symbols are linear combinations of Riesz symbols after changes of variables. 
We find (see \eqref{eq:FractionalLaplacianResolvent} for notations)
\begin{equation}
\label{eq:ResolventEstimateUpperBound}
\| P(\omega,D)^{-1} \|_{L_0^p \to L_0^q} \lesssim \| e_+(\omega,D) \|_{L^p \to L^q} + \| e_-(\omega,D) \|_{L^p \to L^q}
\end{equation}
for $1 \leq p,q \leq \infty$ with $(1 < p < \infty$ or $1<q<\infty)$. The reason we are not required to take $1<p<\infty$ and $1<q<\infty$ is that, if there is one reflexive $L^p$-space, then we can commute the Fourier multipliers after multiplying out the matrices such that the Riesz transforms act on a reflexive $L^p$-space.\footnote{I thank the referee for pointing this out.} This shows the lower bound in \eqref{eq:ResolventEquivalence} for $d=2$.

\medskip

We turn to show the upper bound in \eqref{eq:ResolventEquivalence}, which is
\begin{equation}
\label{eq:EquivalenceHalfLaplacians}
\| P(\omega,D)^{-1} \|_{L_0^p \to L_0^q} \gtrsim \| e_+(\omega,D) \|_{L^p \to L^q} + \| e_-(\omega,D) \|_{L^p \to L^q}
\end{equation}
for $1<p,q<\infty$.

The operators $\mathcal{R}_j^{\varepsilon'}$ satisfy for $1<p<\infty$
\begin{equation}
\label{eq:EquivalenceRieszTransforms}
\| f \|_{L^p(\R^2)} \sim_{p,\varepsilon,\mu} \| \mathcal{R}_1^{\varepsilon^\prime} f \|_{L^p(\R^2)} + \| \mathcal{R}_2^{\varepsilon^\prime} f \|_{L^p(\R^2)}.
\end{equation}  

In fact, as already used above, $\| \mathcal{R}_j^{\varepsilon^\prime} f \|_{L^p} \lesssim_{p,\varepsilon,\mu} \| f \|_{L^p}$ for $1<p<\infty$ as a consequence of Theorem \ref{thm:MultiplierTheorem}. Let $\chi_1$, $\chi_2: \R / (2 \pi \Z) \to [0,1]$ be a smooth partition of unity of the unit circle such that
\begin{equation*}
\left\{ \begin{array}{cl}
\chi_1(\theta) &= 1 \text{ for } \theta \in [-\frac{\pi}{8},\frac{\pi}{8}] \cup [\frac{7 \pi}{8}, \frac{9 \pi}{8}], \\
\chi_2(\theta) &= 1 \text{ for } \theta \in [\frac{3 \pi}{8}, \frac{5 \pi}{8}] \cup [\frac{11 \pi}{8}, \frac{13 \pi}{8}].
\end{array} \right.
\end{equation*}
We extend $\chi_i$ to $\R^2 \backslash 0$ by zero-homogeneity.

\medskip

 For the reverse bound in \eqref{eq:EquivalenceRieszTransforms}, we decompose $f=f_1+f_2$ as $f_i = \chi_i(D) f$. Set $((\mathcal{R}_i^{\varepsilon^\prime})^{-1} f) \widehat (\xi) = \frac{\| \xi \|_{\varepsilon^\prime}}{\xi_i} \hat{f}(\xi)$. Note that $|\xi_i| \gtrsim \| \xi \|_{\varepsilon'}$ for $\xi \in \text{supp}(\hat{f}_i)$. By Theorem \ref{thm:MultiplierTheorem}, we find the estimate
\begin{equation*}
\| \big( \mathcal{R}^{\varepsilon^\prime}_i \big)^{-1} f_i \|_{L^p} \lesssim_{p,\varepsilon,\mu} \| f_i \|_{L^p}.
\end{equation*}
Consequently,
\begin{equation*}
\| f \|_{L^p} \leq \| f_1 \|_{L^p} + \| f_2 \|_{L^p} \leq \sum_{i=1}^2 \| \big( \mathcal{R}^{\varepsilon'}_i \big)^{-1} \mathcal{R}^{\varepsilon'}_i f_i \|_{L^p} \lesssim_{p,\varepsilon,\mu} \sum_{i=1}^2 \| \mathcal{R}^{\varepsilon'}_i f_i \|_p.
\end{equation*}
 
With \eqref{eq:EquivalenceRieszTransforms} in mind, we show \eqref{eq:EquivalenceHalfLaplacians} by considering the  data
\begin{equation}
v =
\begin{pmatrix}
-2 \mathcal{R}_2^{\varepsilon^\prime} f & 2 \mathcal{R}_1^{\varepsilon^\prime} f & 0
\end{pmatrix}^t
.
\end{equation} 
Clearly, $\partial_1 v_1 + \partial_2 v_2 = 0$. We compute
\begin{equation*}
m^{-1}(D) v = \mu
\begin{pmatrix}
0 & 1 & -1
\end{pmatrix}^t
f.
\end{equation*} 
 We further compute
\begin{equation*}
P(\omega,D)^{-1} v = 
\begin{pmatrix}
- \mathcal{R}_2^{\varepsilon^\prime} (e_- + e_+) & \mathcal{R}_1^{\varepsilon'} (e_- + e_+ ) & \mu (- e_- + e_+) 
\end{pmatrix}^t f,
\end{equation*}
and it follows by \eqref{eq:EquivalenceRieszTransforms} 
\begin{equation*}
\begin{split}
\| P(\omega,D)^{-1} v \|_{L^q} &\sim \| (e_-(\omega,D) + e_+(\omega,D)) f \|_{L^q} + \mu \| (e_-(\omega,D) - e_+(\omega,D)) f \|_{L^q} \\
 &\sim \| e_-(\omega,D) f \|_{L^q} + \| e_+(\omega,D) f \|_{L^q}
\end{split}
\end{equation*}
as claimed. Since $\| v \|_{L^p} \sim \| f \|_{L^p}$, by choosing $f$ suitably, we find
\begin{equation*}
\| P(\omega,D)^{-1} \|_{L_0^p \to L_0^q} \gtrsim \max( \| e_- \|_{L^p \to L^q}, \| e_+ \|_{L^p \to L^q} ) \sim \| e_- \|_{L^p \to L^q} + \| e_+ \|_{L^p \to L^q}.
\end{equation*}
Finally, we turn to \eqref{eq:ResolventEstimateAbove}, which reads for $d=2$
\begin{equation}
\label{eq:ResolventEstimateAbove2d}
\begin{split}
&\quad \| P(\omega,D)^{-1} (J_e,J_m) \|_{L^q}  \lesssim (\| e_-(\omega,D) \|_{L^p \to L^q} + \| e_+(\omega,D) \|_{L^p \to L^q}) \|(J_e,J_m)\|_{L^p} \\
&\quad + \| (-\Delta)^{-\frac{1}{2}} \rho_e \|_{L^q}.
\end{split}
\end{equation}
We decompose writing $J = (J_e,J_m)$
\begin{equation*}
P(\omega,D)^{-1} J = (M^2(A,B) \hat{J})^{\vee} + (M_c \hat{J})^{\vee}
\end{equation*}
as in \eqref{eq:Decomposition2d}. The arguments from above estimate the contribution of $(M^2(A,B) \hat{J})^{\vee}$. A computation yields
\begin{equation*}
(M_c \hat{J})(\xi) = 
\begin{pmatrix}
 \varepsilon_{12} \xi_2' -\varepsilon_{22} \xi_1' \\
  \varepsilon_{12} \xi_1' -\varepsilon_{11} \xi_2'\\
0
\end{pmatrix}
\frac{\hat{\rho_e}(\xi)}{\mu \omega \| \xi \|_{\varepsilon'}}
\end{equation*}
with $\rho_e = \partial_1 J_{e1} + \partial_2 J_{e2}$. From this follows
\begin{equation*}
\| (M_c \hat{J})^{\vee} \|_{L^q} \lesssim \| (-\Delta)^{1/2} \rho_e \|_{L^q}
\end{equation*}
by $L^q$-boundedness of $\mathcal{R}_i^{\varepsilon'}$ for $1<q<\infty$ and $\| \xi \|/ \| \xi \|_{\varepsilon'}$ zero-homogeneous and smooth away from the origin. The proof is complete.
\end{proof}

\subsection{Proof of Theorem \ref{thm:ResolventEstimateMaxwell} for $d=3$}
\label{subsection:3d}
We consider $P(\omega,D)$ as in \eqref{eq:Maxwell3DConcise} with $\varepsilon = \text{diag}(\varepsilon_1,\varepsilon_2,\varepsilon_3)$ and $\mu > 0$. Here we consider the partially anisotropic case $a^{-1}= \varepsilon_1 ; \quad \varepsilon_2 = \varepsilon_3 = b^{-1}$ and suppose that $\mu = 1$ without loss of generality, to which we can reduce by linear substitution. The computation also covers the isotropic case $a=b$, which was considered in \cite{CossettiMandel2021}. For $\xi \in \R^3$ we denote
\begin{align*}
\| \xi \|^2 &= \xi_1^2 + \xi_2^2 + \xi_3^2, \quad \| \xi \|^2_\varepsilon = b \xi_1^2 + a \xi_2^2 + a \xi_3^2, \\
\xi' &= \xi / \| \xi \|, \qquad \qquad \quad \; \tilde{\xi} = \xi / \| \xi \|_\varepsilon.
\end{align*}
We write further
\begin{align*}
(\nabla \times u) \widehat{\,} (\xi) = - i \mathcal{B}(\xi) \hat{u}(\xi), \quad \mathcal{B}(\xi) = 
\begin{pmatrix}
0 & \xi_3 & - \xi_2 \\
- \xi_3 & 0 & \xi_1 \\
\xi_2 & -\xi_1 & 0
\end{pmatrix}
.
\end{align*}
We have the following lemma on diagonalization:
\begin{lemma}
\label{lem:DiagonalizationPartiallyAnisotropic}
For almost all $\xi \in \R^3$ there is a matrix $\tilde{m}(\xi) \in \C^{6 \times 6}$ such that
\begin{equation*}
p(\omega,\xi) = \tilde{m}(\xi) d(\omega,\xi) \tilde{m}^{-1}(\xi)
\end{equation*}
with
\begin{equation*}
d(\omega,\xi) = i \, \text{diag}(\omega, \omega, \omega - \sqrt{b} \| \xi \|, \omega + \sqrt{b} \| \xi \|, \omega - \| \xi \|_\varepsilon, \omega + \| \xi \|_\varepsilon).
\end{equation*}
Furthermore, the components of $\tilde{m}$ and $\tilde{m}^{-1}$ are $L^p$-bounded Fourier multipliers for $1<p<\infty$.
\end{lemma}
\begin{proof}
To verify that the diagonal entries of $d$ are truly the eigenvalues of $p$, we record eigenvectors, which are normalized to zero-homogeneous entries. Eigenvectors to $i \omega$ are
\begin{align*}
v_1^t &= \big(0,0,0, \xi_1^\prime , \xi_2^\prime, \xi_3^\prime \big), \\
v_2^t &= \big(\frac{\tilde{\xi}_1}{a }, \frac{\tilde{\xi}_2}{b }, \frac{\tilde{\xi}_3}{b }, 0, 0, 0 \big).
\end{align*}
Eigenvectors to $i\omega \mp i \sqrt{b} \| \xi \|$ are given by
\begin{align*}
v_3^t &= \big(0,- \frac{\xi_3^\prime }{\sqrt{b}}, \frac{\xi_2^\prime }{\sqrt{b}}, - ((\xi_2^\prime)^2 + (\xi_3^\prime)^2), \xi_1^\prime \xi_2^\prime, \xi_1^\prime \xi_3^\prime \big), \\
v_4^t &= \big(0, \frac{\xi_3^\prime }{\sqrt{b} }, - \frac{\xi_2^\prime }{\sqrt{b}}, - ((\xi_2^\prime)^2 + (\xi_3^\prime)^2), \xi_1^\prime \xi_2^\prime, \xi_1^\prime \xi_3^\prime \big).
\end{align*}
Eigenvectors to $i \omega \mp i \| \xi \|_{\varepsilon}$ are given by
\begin{align*}
v_5^t &= \big( \tilde{\xi}_2^2 + \tilde{\xi}_3^2, - \tilde{\xi}_1 \tilde{\xi}_2, - \tilde{\xi}_1 \tilde{\xi}_3, 0 , - \tilde{\xi}_3, \tilde{\xi}_2 \big),\\
v_6^t &= \big(- (\tilde{\xi}_2^2+ \tilde{\xi}_3^2), \tilde{\xi}_1 \tilde{\xi}_2, \tilde{\xi}_1 \tilde{\xi}_3, 0, -\tilde{\xi}_3, \tilde{\xi}_2 \big).
\end{align*}
Set
\begin{equation}
\label{eq:AuxiliaryMatrix}
m(\xi) = (v_1, \ldots, v_6)
\end{equation}
and
\begin{equation}
\label{eq:RenormalizationQuantities}
\alpha(\xi)= \frac{(\xi_2^2 + \xi_3^2)^{1/2}}{(\| \xi \| \| \xi \|_\varepsilon)^{\frac{1}{2}}} \text{ and } \delta = \frac{\| \xi \|}{\| \xi \|_\varepsilon}.
\end{equation}
The determinant of $m(\xi)$ is computed in Lemma \ref{lem:ComputationDeterminant} in the Appendix. We have
\begin{equation*}
| \det m(\xi) | \sim \alpha^4(\xi).
\end{equation*}
Furthermore, we find for $\alpha \neq 0$:
\begin{align*}
&m^{-1}(\xi) = \\
&\begin{pmatrix}
0 & 0 & 0 &  \xi_1^\prime & \xi_2^\prime & \xi_3^\prime \\
ab \tilde{\xi}_1 & ab \tilde{\xi}_2 & ab \tilde{\xi}_3 & 0 & 0 & 0 \\
0 & - \frac{\sqrt{b} \| \xi \|}{2 \| \xi \|_\varepsilon} \frac{\tilde{\xi}_3}{\tilde{\xi}_2^2 + \tilde{\xi}_3^2} &  
\frac{\sqrt{b} \| \xi \|}{2 \| \xi \|_\varepsilon} \frac{\tilde{\xi}_2}{\tilde{\xi}_2^2 + \tilde{\xi}_3^2} & - 1/2 & \frac{\xi_1' \xi_2'}{2(\xi_2'^2 + \xi_3'^2)} & \frac{\xi_1' \xi_3'}{2(\xi_2'^2 + \xi_3'^2)} \\
0 & \frac{\sqrt{b}\| \xi \| }{2 \| \xi \|_\varepsilon} \frac{\tilde{\xi}_3}{\tilde{\xi}_2^2 + \tilde{\xi}_3^2} & - \frac{\sqrt{b} \| \xi \| }{2 \| \xi \|_\varepsilon} \frac{\tilde{\xi}_2}{\tilde{\xi}_2^2 + \tilde{\xi}_3^2} & - 1/2 & \frac{\xi_1^\prime \xi_2^\prime}{2(\xi_2'^2 + \xi_3'^2)} & \frac{\xi_1^\prime \xi_3^\prime}{2(\xi_2'^2 + \xi_3'^2)} \\
a/2 & - \frac{b \tilde{\xi}_1 \tilde{\xi}_2}{2(\tilde{\xi}_2^2 + \tilde{\xi}_3^2)} & - \frac{b \tilde{\xi}_1 \tilde{\xi}_3}{2(\tilde{\xi}_2^2 + \tilde{\xi}_3^2)} & 0 & - \frac{\xi_3^\prime \| \xi \|_\varepsilon}{2 \| \xi \| (\xi_2'^2 + \xi_3'^2)} & \frac{\| \xi \|_\varepsilon \xi_2^\prime }{2 \| \xi \| (\xi_2'^2 + \xi_3'^2)} \\
-a/2 & \frac{b \tilde{\xi}_1 \tilde{\xi}_2}{2(\tilde{\xi}_2^2 + \tilde{\xi}_3^2)} &  \frac{b \tilde{\xi}_1 \tilde{\xi}_3}{2(\tilde{\xi}_2	^2 + \tilde{\xi}_3^2)} & 0 & - \frac{| \xi |_\varepsilon}{2 \| \xi \|} \frac{\xi_3^\prime}{(\xi_2'^2 + \xi_3'^2)} & \frac{\| \xi \|_\varepsilon \xi_2^\prime}{2 \| \xi \| (\xi_2'^2 + \xi_3'^2)}
\end{pmatrix}
.
\end{align*}
Since $\alpha(\xi) \to 0$ as $|\xi_2| + |\xi_3| \to 0$, $m$ becomes singular along the $\xi_1$-axis, and the entries of $m^{-1}(\xi)$ are no $L^p$-bounded Fourier multipliers anymore. This suggests to renormalize $v_3,\ldots,v_6$ with $1/\alpha(\xi)$. We let
\begin{equation*}
\begin{split}
&\tilde{m}(\xi) = \\
&\begin{pmatrix}
0 & \frac{\tilde{\xi}_1}{a} & 0 & 0 & (\delta (\tilde{\xi}_2^2 + \tilde{\xi}_3^2))^{\frac{1}{2}} & - (\delta(\tilde{\xi}_2^2 + \tilde{\xi}_3^2))^{\frac{1}{2}} \\
0 & \frac{\tilde{\xi}_2}{b} & - \frac{\xi_3'}{\sqrt{b} (\delta( \xi_2'^2 + \xi_3'^2))^{\frac{1}{2}}} & \frac{\xi_3'}{\sqrt{b} (\delta(\xi_2'^2 + \xi_3'^2))^{\frac{1}{2}}} & - \frac{\delta^{\frac{1}{2}} \tilde{\xi}_1 \tilde{\xi}_2}{(\tilde{\xi}_2^2 + \tilde{\xi}_3^2)^{1/2}} & \frac{\delta^{\frac{1}{2}} \tilde{\xi}_1 \tilde{\xi}_2}{(\tilde{\xi}_2^2 + \tilde{\xi}_3^2)^{\frac{1}{2}}} \\
0 & \frac{\tilde{\xi}_3}{b} & \frac{\xi_2'}{\sqrt{b} (\delta (\xi_2'^2 + \xi_3'^2))^{\frac{1}{2}}} & - \frac{\xi_2'}{\sqrt{b} (\delta(\xi_2'^2 + \xi_3'^2))^{\frac{1}{2}}} & - \frac{\delta^{\frac{1}{2}} \tilde{\xi}_1 \tilde{\xi}_3}{(\tilde{\xi}_2^2 + \tilde{\xi}_3^2)^{\frac{1}{2}}} & \frac{\delta^{\frac{1}{2}} \tilde{\xi}_1 \tilde{\xi}_3}{(\tilde{\xi}_2^2 + \tilde{\xi}_3^2)^{\frac{1}{2}}} \\
\xi_1' & 0 &- \frac{(\xi_2'^2 + \xi_3'^2)^{\frac{1}{2}}}{\delta^{\frac{1}{2}}} & - \frac{(\xi_2'^2 + \xi_3'^2)^{\frac{1}{2}}}{\delta^{\frac{1}{2}}} & 0 & 0 \\
\xi_2' & 0 & \frac{\xi_1' \xi_2' }{(\delta(\xi_2'^2 + \xi_3'^2))^{\frac{1}{2}}} & \frac{\xi_1' \xi_2' }{(\delta (\xi_2'^2 + \xi_3'^2))^{\frac{1}{2}}} & - \frac{\delta^{\frac{1}{2}} \tilde{\xi}_3}{(\tilde{\xi}_2^2 + \tilde{\xi}_3^2)^{\frac{1}{2}}} & - \frac{\delta^{\frac{1}{2}} \tilde{\xi}_3}{(\tilde{\xi}_2^2 + \tilde{\xi}_3^2)^{\frac{1}{2}}} \\
\xi_3' & 0 & \frac{\xi_1' \xi_3'}{(\delta( \xi_2'^2 + \xi_3'^2))^{\frac{1}{2}}} & \frac{\xi_1' \xi_3' }{ (\delta(\xi_2'^2 + \xi_3'^2))^{\frac{1}{2}}} & \frac{\tilde{\xi}_2 \delta^{\frac{1}{2}}}{(\tilde{\xi}_2^2 + \tilde{\xi}_3^2)^{\frac{1}{2}}} & \frac{\tilde{\xi}_2 \delta^{\frac{1}{2}}}{(\tilde{\xi}_2^2 + \tilde{\xi}_3^2)^{\frac{1}{2}}}
\end{pmatrix}.
\end{split}
\end{equation*}
By Lemma \ref{lem:ComputationDeterminant}, we have $\det (\tilde{m}) \sim 1$ if and only if $\xi \neq (\nu,0,0)$ for some $\nu \in \R$. Hence, $\tilde{m}$ and $\tilde{m}^{-1}$ are well-defined away from the $\xi_1$-axis. By Cramer's rule, we obtain $\tilde{m}(\xi)^{-1}$ from $m^{-1}(\xi)$ by modifying the rows 3-6:
\begin{equation*}
\begin{split}
&\tilde{m}^{-1}(\xi) = \\
&\begin{pmatrix}
0 & 0 & 0 & \xi_1' & \xi_2' & \xi_3' \\
ab \tilde{\xi}_1 & ab \tilde{\xi}_2 & ab \tilde{\xi}_3 & 0 & 0 & 0 \\
0 & - \frac{\sqrt{b} \delta^{\frac{1}{2}} \tilde{\xi}_3}{2 (\tilde{\xi}_2^2+ \tilde{\xi}_3^2)^{\frac{1}{2}}} &  \frac{\sqrt{b} \delta^{\frac{1}{2}} \tilde{\xi}_2}{2 (\tilde{\xi}_2^2 + \tilde{\xi}_3^2)^{\frac{1}{2}}} & - \frac{ (\tilde{\xi}_2^2 + \tilde{\xi}_3^2)^{\frac{1}{2}}}{2 \delta^{\frac{1}{2}}} & \frac{\xi_1' \xi_2' \delta^{\frac{1}{2}}}{2(\xi_2'^2 + \xi_3'^2)^{\frac{1}{2}}} & \frac{\xi_1' \xi_3' \delta^{\frac{1}{2}}}{2(\xi_2'^2 + \xi_3'^2)^{\frac{1}{2}}} \\
0 &  \frac{\sqrt{b} \delta^{\frac{1}{2}} \tilde{\xi}_3}{2 (\tilde{\xi}_2^2+\tilde{\xi}_3^2)^{\frac{1}{2}}} & - \frac{\sqrt{b} \delta^{\frac{1}{2}} \tilde{\xi}_2}{2 (\tilde{\xi}_2^2 + \tilde{\xi}_3^2)^{1/2}} & - \frac{ (\tilde{\xi}_2^2 + \tilde{\xi}_3^2)^{\frac{1}{2}}}{2 \delta^{\frac{1}{2}}} & \frac{\delta^{\frac{1}{2}} \xi_1' \xi_2'}{2(\xi_2'^2 + \xi_3'^2)^{\frac{1}{2}}} & \frac{\delta^{\frac{1}{2}} \xi_1' \xi_3'}{2(\xi_2'^2 + \xi_3'^2)^{\frac{1}{2}}} \\
\frac{a (\tilde{\xi}_2 + \tilde{\xi}_3^2)^\frac{1}{2}}{2 \delta^{\frac{1}{2}}} & - \frac{ b \tilde{\xi}_1 \tilde{\xi}_2}{2 (\delta(\tilde{\xi}_2^2 + \tilde{\xi}_3^2))^{\frac{1}{2}}} & - \frac{b \tilde{\xi}_1 \tilde{\xi}_3}{ 2 (\delta(\tilde{\xi}_2^2 + \tilde{\xi}_3^2))^{\frac{1}{2}}} & 0 & - \frac{\xi_3'}{2(\delta(\xi_2'^2 + \xi_3'^2))^{\frac{1}{2}}} & \frac{\xi_2'}{2  (\delta(\xi_2'^2 + \xi_3'^2))^{\frac{1}{2}}} \\
- \frac{a (\tilde{\xi}_2 + \tilde{\xi}_3^2)^\frac{1}{2}}{2 \delta^{\frac{1}{2}}} & \frac{b \tilde{\xi}_1 \tilde{\xi}_2}{2 (\delta(\tilde{\xi}_2^2 + \tilde{\xi}_3^2))^{\frac{1}{2}}} & \frac{b \tilde{\xi}_1 \tilde{\xi}_3}{2 (\delta(\tilde{\xi}_2^2 + \tilde{\xi}_3^2))^{\frac{1}{2}}} & 0 & - \frac{\xi_3'}{2(\delta (\xi_2'^2 + \xi_3'^2))^{\frac{1}{2}}} & \frac{\xi_2'}{2(\delta( \xi_2'^2 + \xi_3'^2))^{\frac{1}{2}}}
\end{pmatrix}
.
\end{split}
\end{equation*}
Also by Cramer's rule, it is enough to check that the Fourier multipliers associated with the entries in $\tilde{m}$ are $L^p$-bounded, for which we use Theorem \ref{thm:MultiplierTheorem}. 

For the first and second column this is evident since these are Riesz transforms up to change of variables. We turn to the proof that the entries of $v_i/\alpha(\xi)$, $i=3,\ldots,6$, are multipliers bounded in $L^p$ for $1<p<\infty$. This follows by writing them as products of zero-homogeneous functions, which are smooth away from the origin, and Riesz transforms in two variables. We give the details for the entries of $v_3/\alpha(\xi)$:
\begin{itemize}
\item $(v_3)_2/\alpha(\xi)$: We have to show that
\begin{equation*}
\frac{\xi_3 (\| \xi \| \| \xi \|_{\tilde{\varepsilon}})^{1/2} }{\| \xi \| (\xi_2^2 + \xi_3^2)^{1/2}} = \frac{\xi_3}{(\xi_2^2+ \xi_3^2)^{1/2}} \big( \frac{\| \xi \|_{\varepsilon}}{\| \xi \|} \big)^{1/2}
\end{equation*}
is a multiplier. This is the case because $\frac{i \xi_3}{(\xi_2^2+ \xi_3^2)^{1/2}}$ is a Riesz transform in $(x_2,x_3)$ and the second factor $\big( \frac{\| \xi \|_{\varepsilon}}{\| \xi \|} \big)^{1/2}$ is zero-homogeneous and smooth away from the origin, hence, in the scope of Theorem \ref{thm:MultiplierTheorem}.
\item $(v_3)_3/\alpha(\xi)$ is a multiplier by symmetry in $\xi_2$ and $\xi_3$ and the previous considerations.
\item $(v_3)_4/\alpha(\xi)$: We find
\begin{equation*}
\frac{(\xi_2^2 + \xi_3^2)}{\| \xi \|^2 (\xi_2^2+ \xi_3^2)^{1/2}} \cdot (\| \xi \| \| \xi \|_{\varepsilon})^{1/2} = \frac{(\xi_2^2 + \xi_3^2)^{1/2}}{\| \xi \|} \cdot \big( \frac{\| \xi \|_{\varepsilon}}{\| \xi \|} \big)^{1/2}
\end{equation*}
to be a Fourier multiplier as it is zero-homogeneous and smooth away from the origin.
\item $(v_3)_5/\alpha(\xi)$: Consider
\begin{equation*}
\frac{\xi_1 \xi_2}{\| \xi \|^2 (\xi_2^2 + \xi_3^2)^{1/2}} (\| \xi \| \| \xi \|_{\varepsilon})^{1/2} = \frac{\xi_1}{\| \xi \|} \cdot \frac{\xi_2}{(\xi_2^2 + \xi_3^2)^{1/2}} \cdot \big( \frac{\| \xi \|_{\varepsilon}}{\| \xi \|} \big)^{1/2}, 
\end{equation*}
which is again a Fourier multiplier because the first and third expression are zero-homogeneous and smooth in $\R^n \backslash 0$, the second is again a Riesz transform in two variables.
\item $(v_3)_6/\alpha(\xi)$ can be handled like the previous case.
\end{itemize}
The remaining entries of $\tilde{m}$ are treated similarly, which completes the proof.

\end{proof}

\begin{remark}
To compute the eigenvalues from scratch, it is perhaps easiest to use the block structure of $p(\omega,\xi)$ to find
\begin{equation*}
\det (p(\omega,\xi)) = \det (-\omega^2 1_{3 \times 3} - \mathcal{B}^2(\xi) \varepsilon^{-1}).
\end{equation*}
Next, we can use the identity $\mathcal{B}^2(\xi) = - \| \xi \|^2 1_{3 \times 3} + \xi \otimes \xi$, after which there seems to be no further simplification but to compute the determinant brutely. Note that $\det(i \lambda 1_{6 \times 6} - p(\omega,\xi)) = \det (p(\lambda - \omega,\xi))$, which allows to find the eigenvalues from the zero locus of $\det(p(\omega,\xi))$.
\end{remark}

We prove Theorem \ref{thm:ResolventEstimateMaxwell} for $d=3$ following along the argument for $d=2$. Proposition \ref{prop:Explicit3d} in the Appendix provides a decomposition
\begin{equation}
\label{eq:Decomposition3d}
p^{-1}(\omega,\xi)= M^3(A,B,C,D) + M^3_c
\end{equation}
with $M^3_c v = 0$ for $\xi_1 v_1 + \xi_2 v_2 + \xi_3 v_3 = \xi_1 v_4 + \xi_2 v_5 + \xi_3 v_6 = 0$ and
\begin{equation*}
A = \frac{1}{i(\omega - \| \xi \|_\varepsilon)}, \; B = \frac{1}{i(\omega + \| \xi \|_\varepsilon)}, \; C= \frac{1}{i(\omega - \| \xi \|)}, \; D = \frac{1}{i(\omega + \| \xi \|)}.
\end{equation*}

\begin{proof}[Proof of Theorem \ref{thm:ResolventEstimateMaxwell}, $d=3$]
The estimate
\begin{equation*}
\| P(\omega,D)^{-1} \|_{L^p_0(\R^3) \to L^q_0(\R^3)} \lesssim \| ((-\Delta)^{\frac{1}{2}}-\omega)^{-1} \|_{L^p \to L^q} + \| ((-\Delta)^{\frac{1}{2}}-\omega)^{-1} \|_{L^p \to L^q}
\end{equation*}
for $1 \leq p,q \leq \infty$ with $(1 <p < \infty$ or $1<q<\infty)$ follows from the same argument as in the two-dimensional case: The entries of $M^3(A,B,C,D)$ are linear combinations of $A$,$B$,$C$,$D$ multiplied with components of $\tilde{m}$ and $\tilde{m}^{-1}$, which yield Fourier multipliers by Lemma \ref{lem:DiagonalizationPartiallyAnisotropic}.

\medskip

 Below let $(\mathcal{R}_i f) \widehat (\xi) = \frac{\xi_i}{\| \xi \|} \hat{f}(\xi) $.
To show the lower bound for $1<p,q<\infty$, we consider the following initial data:
\begin{equation*}
J_{e} = \begin{pmatrix}
0 \\ - \mathcal{R}_3 f \\ \mathcal{R}_2 f
\end{pmatrix}, \quad
J_{m} = \underline{0}.
\end{equation*}
Note that $\nabla \cdot J_{e} = 0$ and again, the initial data is also physically meaningful as the magnetic current vanishes.

Let $(e_{\pm} f) \widehat (\xi) = (\omega \pm \sqrt{b} | \xi |)^{-1} \hat{f}(\xi)$. We compute with $m$ as in \eqref{eq:AuxiliaryMatrix}:
\begin{equation}
\label{eq:Example}
(d m^{-1})(\xi) \begin{pmatrix}
\hat{J}_{e} \\ \hat{J}_{m}
\end{pmatrix}
= \frac{\sqrt{b}}{2}
\begin{pmatrix}
0 \\ 0 \\  \widehat{e_- f} \\
- \widehat{e_+ f} \\ 0 \\ 0
\end{pmatrix}, \quad
\begin{pmatrix}
D \\ B 
\end{pmatrix}
= i
\begin{pmatrix}
0 \\ - \mathcal{R}_3 ( e_- f + e_+ f) \\ \mathcal{R}_2 ( e_- f + e_+ f) \\
- ((\mathcal{R}_2^2 + \mathcal{R}_3^2) (e_- f - e_+ f) \\ \mathcal{R}_1 \mathcal{R}_2 (e_- f - e_+ f) \\ \mathcal{R}_1 \mathcal{R}_3 (e_- f - e_+ f).
\end{pmatrix}
\end{equation}
We shall see that
\begin{equation}
\label{eq:LowerBound}
\| (D,B) \|_{L^q_0} \gtrsim \| e_- f + e_+ f \|_{L^q} + \| e_- f - e_+ f \|_{L^q} \gtrsim \| e_- f \|_{L^q} + \| e_+ f \|_{L^q}
\end{equation}
either, if $f$ has frequency support in a conic neighbourhood of the $\xi_3$-axis, or, if $f$ is spherically symmetric.

\medskip

\noindent Assume that $g \in \mathcal{S}(\R^3)$ and
\begin{equation*}
\text{supp} (\hat{g}) \subseteq \{ \xi \in \R^3 : | \xi / |\xi| - e_3 | \leq c \ll 1 \text{ and } \frac{1}{2} \leq | \xi | \leq 2 \} =: E.
\end{equation*}
By Theorem \ref{thm:MultiplierTheorem}, we have for $1<p<\infty$
\begin{equation}
\label{eq:RieszTransformConicNeighbourhood}
\| g \|_{L^p} \lesssim \| \mathcal{R}_3 g \|_{L^p} \text{ and } \| \mathcal{R}_2 g \|_{L^p} \leq C(c) \| g \|_{L^p}
\end{equation}
with $C(c) \to 0$ as $c \to 0$. If $\text{supp}( \hat{f}) \subseteq E$, then also the Fourier support of $e_- f \pm e_+ f$ is contained in $E$, and an application of \eqref{eq:RieszTransformConicNeighbourhood} to $D_2$ and $B_1$ yields
\begin{equation}
\label{eq:EstimateBelow}
\begin{split}
\| (D,B) \|_{L^q_0} &\gtrsim \| \mathcal{R}_3 ( e_- f + e_+ f) \|_{L^q} + \| (\mathcal{R}_2^2 + \mathcal{R}_3^2) ( e_- f - e_+ f) \|_{L^q} \\
&\gtrsim \| e_- f + e_+ f \|_{L^q} + \| e_- f - e_+ f \|_{L^q} \\
&\gtrsim \| e_- f \|_{L^q} + \| e_+ f \|_{L^q},
\end{split}
\end{equation}
which is \eqref{eq:LowerBound}.\\
Next, suppose that $f \in L^p(\R^3)$, $1<p<\infty$ is spherically symmetric. Since $\mathcal{R}_1^2 + \mathcal{R}_2^2 + \mathcal{R}_3^2 = Id$ and $\| \mathcal{R}_i^2 f \|_{L^p} = \| \mathcal{R}_j^2 f \|_{L^p}$ for $i,j \in \{1,2,3\}$ by change of variables and rotation symmetry, we find $\| \mathcal{R}_i^2 f \|_{L^p} \gtrsim \| f \|_{L^p}$. By $L^p$-boundedness, we have
\begin{equation}
\label{eq:RieszTransformsEstimateRadialSymmetry}
\| f \|_{L^p} \lesssim \| \mathcal{R}_i^2 f \|_{L^p} \lesssim \| \mathcal{R}_i f \|_{L^p} \lesssim \| f \|_{L^p}.
\end{equation}
Similarly,
\begin{equation*}
(\mathcal{R}_1^2 + \mathcal{R}_2^2) + (\mathcal{R}_2^2 + \mathcal{R}_3^2) + (\mathcal{R}_1^2 + \mathcal{R}_3^2) = 2 Id,
\end{equation*}
and $\| (\mathcal{R}_i^2 + \mathcal{R}_j^2) f \|_{L^p} = \| (\mathcal{R}_k^2 + \mathcal{R}_l^2) f \|_{L^p}$ again by change of variables and rotation symmetry. Hence, we also find
\begin{equation}
\label{eq:RieszTransformBoundBelowSquared}
\| (\mathcal{R}_i^2 + \mathcal{R}_j^2) f \|_{L^p} \gtrsim \| f \|_{L^p}.
\end{equation}
\eqref{eq:RieszTransformsEstimateRadialSymmetry} and \eqref{eq:RieszTransformBoundBelowSquared} together allow to argue as well in case of spherical symmetry as in \eqref{eq:EstimateBelow}. If we can choose $f$ such that the operator norms of $e_{\pm}$ are approximated, we find
 \begin{equation*}
 \| (D,H) \|_{L^q_0} \gtrsim (\| e_- \|_{L^p \to L^q} + \| e_+ \|_{L^p \to L^q} ) \| f \|_{L^p}.
 \end{equation*}
Lastly, if $\text{supp} (\hat{f}) \subseteq E$, i.e., the frequency support is in a conic neighbourhood of the $\xi_3$-axis, or is spherically symmetric, we find $\| (J_{e},J_{m}) \|_{L^p_0} \sim \| f \|_{L^p}$. To see that it suffices to consider the frequency support of $f$ as such, we recall the examples from \cite[Section~5.2]{KwonLee2020}, giving the claimed lower bound for the operator norm of the resolvent of the fractional Laplacian: a Knapp type example, which can be realized with frequency support in a conic neighbourhood of the $\xi_3$-axis \cite[p.~1458]{KwonLee2020}, and a spherically symmetric example related with the surface measure on the sphere \cite[p.~1459]{KwonLee2020}.

\medskip

We turn to the proof of \eqref{eq:ResolventEstimateAbove} for $d=3$:
\begin{equation}
\label{eq:ResolventEstimateAbove3d}
\begin{split}
&\quad \| P(\omega,D)^{-1} (J_e,J_m) \|_{L^q}  \\
&\lesssim (\| ((-\Delta)^{\frac{1}{2}} - \omega)^{-1} \|_{L^p \to L^q} + \| ((-\Delta)^{\frac{1}{2}} + \omega)^{-1} \|_{L^p \to L^q}) \|(J_e,J_m)\|_{L^p} \\
&\quad + \| (-\Delta)^{-\frac{1}{2}} \rho_e \|_{L^q} + \| (-\Delta)^{-\frac{1}{2}} \rho_m \|_{L^q}.
\end{split}
\end{equation}
This hinges again on the decomposition
\begin{equation*}
(P^{-1}(\omega,D) (J_e,J_m))^{\wedge}(\xi) = M^3(A,B,C,D) (\hat{J}_e,\hat{J}_m)(\xi) + M^3_c  (\hat{J}_e,\hat{J}_m)(\xi).
\end{equation*}
The contribution of $M^3(A,B,C,D)$ is estimated like in the first part of the proof. We compute
\begin{equation*}
\begin{split}
&M^3_c (\hat{J}_e,\hat{J}_m)(\xi) \\
&= -
(\frac{b \tilde{\xi}_1 \hat{\rho}_e(\xi)}{ \omega \| \xi \|_{\varepsilon}} , \frac{a \tilde{\xi}_2 \hat{\rho}_e(\xi)}{ \omega \| \xi \|_{\varepsilon}} , \frac{a \tilde{\xi}_3 \hat{\rho}_e(\xi)}{ \omega \| \xi \|_{\varepsilon}} , \frac{\xi_1' \hat{\rho}_m(\xi)}{  \omega \| \xi \|} , \frac{ \xi_2' \hat{\rho}_m(\xi)}{ \omega \| \xi \|},  \frac{\xi_3' \hat{\rho}_m(\xi)}{ \omega \| \xi \|} )^t
.
\end{split}
\end{equation*}
The claim follows by Theorem \ref{thm:MultiplierTheorem} because $\| \xi \| / \| \xi \|_\varepsilon$ and $\xi_i'$ and $\tilde{\xi}_i$ are zero-\-homo\-geneous and smooth away from the origin. The proof of Theorem \ref{thm:ResolventEstimateMaxwell} is complete.
\end{proof}

\section{Local and global LAP}
\label{section:LAP}
Let $P(\omega,D)$ be as in the previous section. In the following we want to investigate the limit of
\begin{equation*}
P(\omega \pm i \delta, D)^{-1} f \text{ as } \delta \to 0, \quad \omega \in \R \backslash 0,
\end{equation*}
by which we construct solutions to time-harmonic Maxwell's equations. By scaling we see that the following estimates are uniform in $\omega$, provided it varies in a compact set away from the origin. We further suppose that $\omega > 0$; the case $\omega < 0$ can be treated with the obvious modifications.


In the following let $0<|\delta|<1/2$. By the above diagonalization, it is equivalent to consider uniform boundedness of
\begin{equation*}
e^{\varepsilon^\prime}_{\pm}(\omega +i \delta): L^p(\R^d) \to L^q(\R^d), \quad (e^{\varepsilon'}_{\pm}(\omega + i \delta) f) \widehat (\xi) = \frac{\hat{f}(\xi)}{\|\xi\|_{\varepsilon^\prime} \pm ( \omega+ i \delta)}.
\end{equation*}
Hence, by the results of the previous section, the uniform $L^p_0$-$L^q_0$-LAP fails due to the lack of uniform resolvent estimates for the Half-Laplacian in $L^p$-spaces. This is recorded in Corollary \ref{cor:GlobalLAP}.

Regarding the local $L^p_0$-$L^q_0$-LAP, we observe that the operator
\begin{equation*}
(e^{\varepsilon'}_{+}(\omega \pm i \delta) f) \widehat (\xi) = \frac{\beta(\xi) \hat{f}(\xi)}{\|\xi\|_{\varepsilon'} + (\omega \pm i \delta)}
\end{equation*}
for $\beta \in C^\infty_c$, $0<\delta<1/2$ is bounded from $L^p \to L^q$ for $1 \leq p \leq q \leq \infty$ by Young's inequality, with the obvious limit as $\delta \to 0$. Thus, we focus on
\begin{equation}
\label{eq:ReducedOperator}
(e_\delta f) \widehat (\xi) := (e_-(\omega \pm i\delta) f) \widehat (\xi) = \frac{\beta(\xi) \hat{f}(\xi)}{\|\xi\|_{\varepsilon'} - (\omega \pm i \delta)}
\end{equation}
with $0< \delta  < \delta_0 \ll 1$, where $\beta \in C^\infty_c(\R^n)$.

We can be more precise about the limiting operators: For $t \in \R$ recall Sokhotsky's formula, which hold in the sense of distributions:
\begin{equation*}
\lim_{\varepsilon \downarrow 0} \frac{1}{t \pm i \varepsilon} = v.p. \frac{1}{t} \mp i \pi \delta_0(t),
\end{equation*}
where $\delta_0$ denotes the delta-distribution at the origin.\\
Let
\begin{equation*}
\mathcal{R}_{\pm}^{loc} f = \lim_{\delta \to \pm 0} e_{\delta} f.
\end{equation*}
We find
\begin{equation*}
\mathcal{R}_{\pm}^{loc} f = v.p. \int \frac{\beta(\xi) e^{ix\xi}}{\|\xi\|_{\varepsilon'} - \omega} \hat{f}(\xi) d\xi \pm i \pi \int e^{i x \xi} \beta(\xi) \delta(\|\xi\|_{\varepsilon'} - \omega) \hat{f}(\xi) d\xi,
\end{equation*}
and by the diagonalization formulae, we find that the limiting operators can be expressed as linear combinations involving possibly generalized Riesz transforms, $\mathcal{R}^{loc}_{\pm}$, and $e_+$. We recall the $L^p$-$L^q$-mapping properties of $\mathcal{R}^{loc}_{\pm}$.

We observe that
\begin{equation*}
(\mathcal{R}^{loc}_+ - \mathcal{R}^{loc}_-) f = 2 \pi i \int_{\{ \| \xi \|_{\varepsilon'} = 1 \}} \beta(\xi) e^{ix \xi} \hat{f}(\xi) d\sigma(\xi).
\end{equation*}
This operator, modulo the bounded operator given by convolution with $\mathcal{F}^{-1} \beta$ and linear change of variables  $\xi \to \zeta$ such that $\| \xi \|_{\varepsilon'} = \| \zeta \|$, is known as \emph{restriction-extension operator} (cf. \cite{JeongKwonLee2016,KwonLee2020}) and is a special case of the Bochner-Riesz operator of negative index:
\begin{equation*}
(\mathcal{B}^{\alpha} f) \widehat (\xi) = \frac{1}{\Gamma(1-\alpha)} \frac{\hat{f}(\xi)}{(1-\| \xi \|^2)_+^\alpha}, \quad 0 < \alpha \leq \frac{d+2}{2}, 
\end{equation*}
$\mathcal{B}_\alpha$ is defined by analytic continuation for $\alpha \geq 1$. Hence, for $\alpha = 1$, it matches the restriction--extension operator. This operator is well-understood due to the works of B\"orjeson \cite{Boerjeson1986}, Sogge \cite{Sogge1986}, and Guti\'errez \cite{Gutierrez2000,Gutierrez2004}. The most recent results for Bochner--Riesz operators of negative index are due to Kwon--Lee \cite{KwonLee2020}. Guti\'errez showed that $\mathcal{B}^1: L^p \to L^q$ is bounded if and only if $(1/p,1/q) \in \mathcal{P}(d)$ with
\begin{equation*}
\mathcal{P}(d) = \{ (x,y) \in [0,1]^2 \, : \, x-y \geq \frac{2}{d+1}, \; x > \frac{d+1}{2d}, \; y < \frac{d-1}{2d} \}.
\end{equation*}
She used this to show uniform resolvent estimates for 
\begin{equation*}
(-\Delta - z)^{-1}: L^p \to L^q, \quad z \in \mathbb{S}^1 \backslash \{1 \} \text{ for } (1/p,1/q) \in \mathcal{R}_1(d).
\end{equation*}
We summarize the operator bounds for $e_\delta$ and $\mathcal{R}_{\pm}^{\text{loc}}$.
\begin{proposition}[{\cite[Proposition~4.1]{KwonLee2020}}]
\label{prop:OperatorBounds}
Let $\omega > 0$, $0<\delta<1/2$, $\beta \in C^\infty_c(\R^d)$ and $d_\delta$ as in \eqref{eq:ReducedOperator}. Then, we find the following estimates to hold for $(1/p,1/q) \in \mathcal{P}(d)$:
\begin{equation}
\label{eq:OperatorBounds}
\begin{split}
\| e_\delta \|_{L^p \to L^q} &\leq C(\omega,p,q), \\
\big\| \int_{\R^d} e^{ix.\xi} \delta(\| \xi \|_{\varepsilon'} - \omega) \hat{f}(\xi) d\xi \big\|_{L^q} &\leq C(\omega,p,q) \| f \|_{L^p}, \\
\big\| v.p. \int_{\R^d} e^{ix.\xi} \frac{\beta(\xi)}{\| \xi \|_{\varepsilon'} - \omega} \hat{f}(\xi) d\xi \big\|_{L^q} &\leq C(\omega,p,q) \|f \|_{L^p}.
\end{split}
\end{equation}
\end{proposition}
We are ready for the proof of the local LAP:
\begin{proposition}[Local LAP]
\label{prop:LocalLAP}
We find a local $L^p_0$-$L^q_0$-LAP to hold provided that $(1/p,1/q) \in \mathcal{P}(d)$. This means that for $\omega \in \R \backslash 0$ and $\beta \in C^\infty_c(\R^d)$, we find uniform (in $0<\delta<1/2$) resolvent bounds
\begin{equation}
\label{eq:UniformResolventBounds}
\| P(\omega \pm i \delta,D)^{-1} \beta(D) f \|_{L^q_0(\R^d)} \lesssim_{p,q,d,\omega} \| f \|_{L^p_0(\R^d)}
\end{equation}
and there are limiting operators $P_{\pm}^{loc}: L^p_0 \to L^q_0$ such that
\begin{equation*}
P(\omega \pm i \delta, D)^{-1} \beta(D) f \to P_{\pm}^{loc}(\omega) f \text{ in } (\mathcal{S}'(\R^d))^{m(d)}.
\end{equation*}
\end{proposition}
\begin{proof}
We assume that $\omega > 0$ because $\omega < 0$ can be treated \emph{mutatis mutandis}. Recall the bounds for $e_\delta$ recorded in Proposition \ref{prop:OperatorBounds}, easier bounds for $e_+^{\varepsilon'}$, and the diagonalization from Sections \ref{section:Reduction}, which decompose (cf. Lemmas \ref{lem:Diagonalization2d}, \ref{lem:DiagonalizationPartiallyAnisotropic})
\begin{equation*}
p(\omega,\xi) = m(\xi) d(\omega,\xi) m^{-1}(\xi).
\end{equation*}
By these, \eqref{eq:UniformResolventBounds} follows for $(1/p,1/q) \in \mathcal{P}(d)$ provided that $1<p,q<\infty$ to bound the generalized Riesz transforms. We extend this to all $(1/p,1/q) \in \mathcal{P}(d)$ by Young's inequality: For $(1/p,0) \in \mathcal{P}(d)$ we choose $1<\tilde{q}<\infty$ such that $(1/p,1/\tilde{q})  \in \mathcal{P}(d)$. By Young's inequality and the previously established bounds for $(1/p,1/\tilde{q}) \in \mathcal{P}(d)$ follows
\begin{equation*}
\| P(\omega \pm i \delta,D)^{-1} \beta(D) f \|_{L^\infty_0} \lesssim \| P(\omega \pm i \delta,D)^{-1} \beta(D) f \|_{L^{\tilde{q}}_0} \lesssim \| f \|_{L^p_0}.
\end{equation*}
The case $(1,1/q) \in \mathcal{P}(d)$ is treated by the dual argument.

\medskip

 By Sokhotsky's formula and the diagonalization, we can consider the limiting operators
\begin{equation*}
P_{\pm}(\omega,D) = \lim_{\delta \to 0} P(\omega\pm i \delta,D)^{-1} \beta(D): L^p_0 \to L^q_0
\end{equation*}
whose mapping properties follow again from Proposition \ref{prop:OperatorBounds} and the diagonalization as argued above. We give explicit formulae in Propositions \ref{prop:Explicit2d} and \ref{prop:Explicit3d}; however, these are bulky and recorded in the appendix.
\end{proof}

We are ready for the proof of Theorem \ref{thm:LocalLAP}:
\begin{proof}[Proof of Theorem \ref{thm:LocalLAP}]

 Let $1 \leq p_1, p_2, q \leq \infty$, and $\omega \in \R \backslash 0$. Choose $C=C(\varepsilon,\omega)$ such that $p(\omega,\xi)^{-1}$ is regular for $\| \xi \| \geq C$. Write $J = (J_e,J_m)$ for the sake of brevity. Let $\beta \in C^\infty_c$ with $\beta \equiv 1$ on $\{ \| \xi \| \leq C \}$ and decompose
\begin{equation*}
J
= \beta(D) J + (1-\beta)(D) J
=: J_{low} + J_{high}.
\end{equation*}
By Proposition \ref{prop:LocalLAP}, we find uniform bounds for $0 < \delta < 1/2$
\begin{equation*}
\| P(\omega \pm i \delta, D)^{-1} J_{low} \|_{L_0^q} \lesssim \| J_{low} \|_{L_0^{p_1}}
\end{equation*}
provided that $(\frac{1}{p_1},\frac{1}{q}) \in \mathcal{P}(d)$. The estimate
\begin{equation*}
\| P(\omega \pm i \delta,D)^{-1} J_{high} \|_{L_0^{q}} \lesssim \| J_{high} \|_{L_0^{p_2}}
\end{equation*}
follows for $0 \leq \frac{1}{p_2} - \frac{1}{q} \leq \frac{1}{d}$ and $(\frac{1}{p_2},\frac{1}{q}) \notin \{ (\frac{1}{d},0), (1,\frac{d-1}{d}) \}$ by properties of the Bessel kernel. The limiting operators $P^{loc}_{\pm}(\omega)$ were described in Proposition \ref{prop:LocalLAP}: We have
\begin{equation*}
P(\omega \pm i \delta, D)^{-1} (J_e,J_m) \to P^{loc}_{\pm}(\omega) (J_e,J_m) \text{ in } \mathcal{S}'(\R^d)^{m(d)}.
\end{equation*}
The high frequency is limit is easier to analyze because the multiplier remains regular by construction. Let $M^d \in \C^{m(d) \times m(d)}$ be as in Propositions \ref{prop:Explicit2d} and \ref{prop:Explicit3d}. For $d=2$, let
\begin{equation*}
A = \frac{1}{i(\omega - \| \xi \|_{\varepsilon'})}, \quad B= \frac{1}{i(\omega + \| \xi \|_{\varepsilon'})},
\end{equation*}
and we have
\begin{equation}
\label{eq:2dHighFrequencyLimit}
\begin{split}
P(\omega \pm i \delta,D)^{-1} J_{high} &\to \frac{1}{(2 \pi)^2} \int_{\R^2} e^{ix.\xi} M^2(A,B) (1-\beta(\xi)) \hat{J}(\xi) d\xi \text{ in } (\mathcal{S}'(\R^2))^3 \\
 &=: P^{high}(\omega) J.
\end{split}
\end{equation}
For $d=3$, let
\begin{equation*}
A = \frac{1}{i(\omega - \sqrt{b} \| \xi \|)}, \; B = \frac{1}{i(\omega + \sqrt{b} \| \xi \|)}, \; C = \frac{1}{i(\omega - \| \xi \|_\varepsilon)}, \; D = \frac{1}{i(\omega + \| \xi \|_\varepsilon)},
\end{equation*}
and we have with convergence in $ (\mathcal{S}'(\R^3))^6$ 
\begin{equation}
\label{eq:3dHighFrequencyLimit}
\begin{split}
P(\omega \pm i \delta,D)^{-1} (1-\beta(D)) J &\to \frac{1}{(2 \pi)^3} \int_{\R^3} e^{ix.\xi} M^3(A,B,C,D) (1-\beta(\xi)) \hat{J}(\xi) d\xi \\
 &=: P^{high}(\omega) J.
 \end{split}
\end{equation}
Let $P_{\pm}(\omega) = P_{\pm}^{loc}(\omega) + P^{high}(\omega)$. By Proposition \ref{prop:LocalLAP}, and \eqref{eq:2dHighFrequencyLimit}, \eqref{eq:3dHighFrequencyLimit}, we have
\begin{equation*}
P(\omega \pm i \delta,D)^{-1} J \to P_{\pm}^{loc}(\omega) J + P^{high}(\omega) J \text{ in } (\mathcal{S}'(\R^d))^{m(d)}.
\end{equation*}

Let $(D,B)^{\pm}_\delta = P(\omega \pm i \delta, D)^{-1} J$ and $(D,B)^{\pm} = P_{\pm}(\omega) J$. At last, we show that
\begin{equation}
\label{eq:LimitingSolution}
P(\omega,D)(D,B)^{\pm} = J.
\end{equation}
For this purpose, we show that for $\delta \to 0$ we have
\begin{equation}
\label{eq:Limit}
P(\omega,D) (D,B)^{\pm}_\delta \to J \text{ in } \mathcal{S}'(\R^d)^{m(d)}.
\end{equation}
As $(D,B)^{\pm}_\delta \to (D,B)^{\pm}$ in $\mathcal{S}'(\R^d)^{m(d)}$, \eqref{eq:Limit} concludes the proof.\\
To show \eqref{eq:Limit}, we return to the diagonalizations (cf. Lemmas \ref{lem:Diagonalization2d}, \ref{lem:DiagonalizationPartiallyAnisotropic}):
\begin{equation*}
p(\tilde{\omega},\xi) = i m(\xi) d(\tilde{\omega},\xi) m^{-1}(\xi) \text{ for } \tilde{\omega} \in \C.
\end{equation*}
We find for $\omega \in \R$:
\begin{equation*}
\begin{split}
p(\omega, \xi) p^{-1}(\omega \pm i \delta, \xi) &= m(\xi) d(\omega,\xi) d(\omega \pm i \delta, \xi)^{-1} m^{-1}(\xi) \\
&= m(\xi) ( 1_{m(d) \times m(d)} \pm \delta d(\omega \pm i \delta,\xi)^{-1} ) m^{-1}(\xi) \\
&= 1_{m(d) \times m(d)} \pm \delta p(\omega \pm i \delta, \xi)^{-1}.
\end{split}
\end{equation*}
Hence,
\begin{equation*}
\begin{split}
P(\omega,D) (D,B)^{\pm}_\delta &= J \pm \delta P(\omega \pm i \delta, D)^{-1} J, \\
\| P(\omega, D) (D,B)^{\pm}_\delta - J \|_{L^q_0(\R^d)} &\lesssim \delta \| J \|_{L_0^{p_1} \cap L_0^{p_2}} \to 0.
\end{split}
\end{equation*}
In particular, \eqref{eq:Limit} holds true in $\mathcal{S}'(\R^d)^{m(d)}$.

\medskip

Next, we suppose additionally that $J \in (W^{s-1,q}(\R^d))^{m(d)}$ for $s \geq 1$. By Young's inequality, we have
\begin{equation*}
\| P(\omega \pm i \delta,D)^{-1} \beta(D) J \|_{(W^{s,q}(\R^d))^{m(d)}} \lesssim \| P(\omega \pm i \delta, D)^{-1} \beta(D) J \|_{L^{q}_0(\R^d)}.
\end{equation*}
Hence, the low frequencies can be estimated like before. For the high frequencies, we recall that the multipliers $M^2$ and $M^3$ yield smoothing of one derivative and by Theorem \ref{thm:MultiplierTheorem}, we find
\begin{equation*}
\begin{split}
&\qquad \| P(\omega \pm i \delta, D)^{-1} (1-\beta(D)) J \|_{(W^{s,q}(\R^d))^{m(d)}} \\
&\lesssim \| \frac{(1-\Delta)^{s/2}}{(1-\Delta)^{1/2}} (1-\beta(D)) J \|_{(L^q(\R^d))^{m(d)}} \\
&\lesssim \| (1-\Delta)^{(s-1)/2} (1-\beta(D)) J \|_{(L^q(\R^d))^{m(d)}} \\
 &= \| (1-\beta(D)) J \|_{(W^{s-1,q}(\R^d))^{m(d)}}.
\end{split}
\end{equation*}
The proof of Theorem \ref{thm:LocalLAP} is complete.
\end{proof}

\section{Localization of Eigenvalues}
\label{section:Localization}
At last, we use the $\omega$-dependent resolvent estimates to localize eigenvalues for operators $P(\omega,D) + V$ acting in $L^q$. For this purpose, we consider for $\ell > 0$ and $(1/p,1/q) \in \tilde{\mathcal{R}}_0^{\frac{1}{2}}$ the region, where uniform resolvent estimates are possible:
\begin{equation}
\label{eq:UniformEstimate}
\begin{split}
\mathcal{Z}_{p,q}(\ell) &= \{ \omega \in \C \backslash \R \; : \; \kappa_{p,q}(\omega) \leq \ell \} \\
&= \{ \omega \in \C \backslash \R \; : \; |\omega|^{-\alpha_{p,q}} |\omega|^{\gamma_{p,q}} | \Im \omega |^{-\gamma_{p,q}} \leq \ell \}, \quad \alpha_{p,q} = 1 - d \big( \frac{1}{p} - \frac{1}{q} \big).
\end{split}
\end{equation}
Describing the regions, we start with observing the symmetry in the real and imaginary part. For $\alpha_{p,q} = 0$, $\ell < 1$, we find $\mathcal{Z}_{p,q}(\ell) = \emptyset$. For $\ell \geq 1$, $\mathcal{Z}_{p,q}(\ell)$ describes a cone around the $y$-axis with aperture getting larger. For $\alpha_{p,q} > 0$ the boundaries become slightly curved.
Pictorial representations for $\Re \omega > 0$ were provided in \cite[Figures~9~(a)-(c)]{KwonLee2020}. The region in the left half plane is obtained by reflection along the imaginary axis. We shall see that eigenvalues of $P(\omega,D) + V$ must lie in $\C \backslash \mathcal{Z}_{p,q}(\ell)$. Previously in \cite{Frank2018}, for non-self-adjoint Schr\"odinger operators analogous arguments were used to show that in a range of $(p,q)$, a sequence of eigenvalues $\lambda_j$ with $\Re \lambda_j \to \infty$ has to satisfy $\Im \lambda_j \to 0$ as a consequence of the shape of $\mathcal{Z}_{p,q}(\ell)$. This is not the case presently and the shape of $\mathcal{Z}_{p,q}(\ell)$ only yields a bound for the asymptotic growth of $|\Im \lambda_j|$ as $|\Re \lambda_j| \to \infty$. This also raises the question for counterexamples, where the behavior $\Re \lambda_j \to \infty$ and $\Im \lambda_j \to 0$ fails. We also refer to Cuenin \cite{Cuenin2017} for resolvent estimates for the fractional Laplacian in this context.

Let $C$ be the constant such that
\begin{equation}
\label{eq:ConstantResolventEstimate}
\| P(\omega,D)^{-1} \|_{L_0^p(\R^d) \to L_0^q(\R^d)} \leq C \kappa_{p,q}(\omega).
\end{equation}

\begin{corollary}
Let $d \in \{2,3\}$, $\ell > 0$, and $1<p,q<\infty$ such that $(1/p,1/q) \in \tilde{\mathcal{R}}_0^{1/2}$. Suppose that there is $t \in (0,1)$ such that
\begin{equation*}
\| V \|_{\frac{pq}{q-p}} \leq t (C \ell)^{-1}.
\end{equation*}
If $E \in \C \backslash \R$ is an eigenvalue of $P+V$ acting in $L^q_0$, then $E$ must lie in $\C \backslash \mathcal{Z}_{p,q}(\ell)$.
\end{corollary}
\begin{proof}
The short argument is standard by now (cf. \cite{KwonLee2020,KwonLeeSeo2021}), but contained for the sake of completeness. Let $u \in L_0^q(\R^d)$ be an eigenfunction of $P+V$ with eigenvalue $E \in \C \backslash \R$ and suppose that $E \in \mathcal{Z}_{p,q}(\ell)$. By H\"older's inequality, we find $-(P-E)u = (V-(P-E+V))u = Vu \in L^p$. By definition of $\mathcal{Z}_{p,q}(\ell)$, we find
\begin{equation*}
\| (P-E)^{-1} \|_{p \to q} \leq C \kappa_{p,q}(E) \leq C \ell.
\end{equation*}
By the triangle and H\"older's inequality, we find
\begin{equation*}
\|(P-E)^{-1}(P-E) u \|_q \leq C \ell ( \| (P-E+V) u \|_p + \| V u \|_p ) \leq C \ell \| V \|_{\frac{pq}{q-p}} \| u \|_q \leq t \| u \|_q,
\end{equation*}
which implies $u=0$ as $t<1$. Hence, $E \notin \mathcal{Z}_{p,q}(\ell)$.
\end{proof}

\section{Appendix}
\begin{lemma}
\label{lem:ComputationDeterminant}
With the notations from Section \ref{subsection:3d}, let $m(\xi)$ be as in \eqref{eq:AuxiliaryMatrix} and $\alpha(\xi)$ as in \eqref{eq:RenormalizationQuantities}. Then, we have
\begin{equation*}
|\det m(\xi)| \sim \alpha^4(\xi).
\end{equation*}
\end{lemma}
\begin{proof}
We compute the determinant by taking linear combinations of the third and fourth column and fifth and sixth column, and aligning the columns as block matrices:
\begin{equation*}
\det m(\xi) = 
\begin{vmatrix}
0 & \tilde{\xi}_1/a & 0 & 0 & \tilde{\xi}_2^2 + \tilde{\xi}_3^2 & -(\tilde{\xi}_2^2 + \tilde{\xi}_3^2) \\
0 & \tilde{\xi}_2/b & - \xi_3'/\sqrt{b} & \xi_3'/\sqrt{b} & - \tilde{\xi}_1 \tilde{\xi}_2 & \tilde{\xi}_1 \tilde{\xi}_2 \\
0 & \tilde{\xi}_3/b & \xi_2'/\sqrt{b} & - \xi_2'/\sqrt{b} & - \tilde{\xi}_1 \tilde{\xi}_3 & \tilde{\xi}_1 \tilde{\xi}_3 \\
\xi_1' & 0 & -((\xi_2')^2 + (\xi_3')^2) & - ((\xi_2')^2 + (\xi_3')^2) & 0 & 0 \\
\xi_2' & 0 & \xi_1' \xi_2' & \xi_1' \xi_2' & - \tilde{\xi}_3 & - \tilde{\xi}_3 \\
\xi_3' & 0 & \xi_1' \xi_3' & \xi_1' \xi_3' & \tilde{\xi}_2 & \tilde{\xi}_3
\end{vmatrix}
\end{equation*}
\begin{equation*}
\begin{split}
&\sim 
\begin{vmatrix}
0 & \tilde{\xi}_1/a & 0 & 0 & \tilde{\xi}_2^2 + \tilde{\xi}_3^2 & 0 \\
0 & \tilde{\xi}_2/b & 0 & - \xi_3'/\sqrt{b} & - \tilde{\xi}_1 \tilde{\xi}_2 & 0 \\
0 & \tilde{\xi}_3/b & 0 & - \xi_2'/\sqrt{b} & - \tilde{\xi}_1 \tilde{\xi}_3 & 0 \\
\xi_1' & 0 & (\xi_2')^2 + (\xi_3')^2 & 0 & 0 & 0 \\
\xi_2' & 0 & -\xi_1' \xi_2' & 0 & 0 & - \tilde{\xi}_3 \\
\xi_3' & 0 & -\xi_1' \xi_3' & 0 & 0 & \tilde{\xi}_3
\end{vmatrix}
\\
&\sim 
\begin{vmatrix}
\tilde{\xi}_1/a & 0 & \tilde{\xi}_2^2+ \tilde{\xi}_3^2 & 0 & 0 & 0 \\
\tilde{\xi}_2/b & - \xi_3'/\sqrt{b} & - \tilde{\xi}_1 \tilde{\xi}_2 & 0 & 0 & 0 \\
\tilde{\xi}_3/b & \xi_2'/\sqrt{b} & - \tilde{\xi}_1 \tilde{\xi}_3 & 0 & 0 & 0 \\
0 & 0 & 0 & \xi_1' & (\xi_2')^2 + (\xi_3')^2 & 0 \\
0 & 0 & 0 & \xi_2' & - \xi_1' \xi_2' & - \tilde{\xi}_3 \\
0 & 0 & 0 & \xi_3' & - \xi_1' \xi_3' & \tilde{\xi}_2 
\end{vmatrix}
=: A_2 \cdot A_1.
\end{split}
\end{equation*}
We find by noting that $(\xi_1')^2 + (\xi_2')^2 + (\xi_3')^2 = 1$
\begin{equation*}
\begin{split}
A_1 &\sim 
\begin{vmatrix}
(\xi_2')^2 + (\xi_3')^2 & - \xi_1' \xi_2' & - \xi_1' \xi_3' \\
\xi_1' & \xi_2' & \xi_3' \\
0 & - \tilde{\xi}_3 & \tilde{\xi}_2
\end{vmatrix}
=
\begin{vmatrix}
1-(\xi_1')^2 & - \xi_1' \xi_2' & - \xi_1' \xi_3' \\
\xi_1' & \xi_2' & \xi_3' \\
0 & - \tilde{\xi}_3 & \tilde{\xi}_2
\end{vmatrix}
\\
&=
\begin{vmatrix}
1 & 0 & 0 \\
\xi_1' & \xi_2' & \xi_3' \\
0 & -\tilde{\xi}_3 & \tilde{\xi}_2
\end{vmatrix}
-\xi_1'
\begin{vmatrix}
\xi_1' & \xi_2' & \xi_3' \\
\xi_1' & \xi_2' & \xi_3' \\
0 & - \tilde{\xi}_3 & \tilde{\xi}_2
\end{vmatrix}
= \xi_2' \tilde{\xi}_2 + \xi_3' \tilde{\xi}_3.
\end{split}
\end{equation*}
Next, by a similar argument,
\begin{equation*}
\begin{split}
A_2 &\sim
\begin{vmatrix}
\tilde{\xi}_1/a & \tilde{\xi}_2/b & \tilde{\xi}_3/b \\
0 & - \xi_3'/\sqrt{b} & \xi_2'/\sqrt{b} \\
\tilde{\xi}_2^2 + \tilde{\xi}_3^2 & - \tilde{\xi}_1 \tilde{\xi}_2 & - \tilde{\xi}_1 & \tilde{\xi}_3 
\end{vmatrix}
= \frac{1}{a}
\begin{vmatrix}
\tilde{\xi}_1/a & \tilde{\xi}_2/b & \tilde{\xi}_3/b \\
0 & - \xi_3'/\sqrt{b} & \xi_2'/\sqrt{b} \\
a(\tilde{\xi}_2^2 + \tilde{\xi}_3^2) & - a \tilde{\xi}_1 \tilde{\xi}_2 & - a \tilde{\xi}_1 \tilde{\xi}_3
\end{vmatrix}
\\
&= \frac{1}{a}
\begin{vmatrix}
\tilde{\xi}_1/a & \tilde{\xi}_2/b & \tilde{\xi}_3/b \\
0 & - \xi_3'/\sqrt{b} & \xi_2'/\sqrt{b} \\
b \tilde{\xi}_1^2 + a(\tilde{\xi}_2^2 + \tilde{\xi}_3^2) - b \tilde{\xi}_1^2 & - a \tilde{\xi}_1 \tilde{\xi}_2 & - a \tilde{\xi}_1 \tilde{\xi}_3 
\end{vmatrix}.
\end{split}
\end{equation*}
We use multilinearity to write
\begin{equation*}
\begin{split}
A_2 &\sim \frac{1}{a}
\left(
\begin{vmatrix}
\tilde{\xi}_1/a & \tilde{\xi}_2/b & \tilde{\xi}_3/ b \\
0 & - \xi_3'/\sqrt{b} & \xi'_2/\sqrt{b} \\
1 & 0 & 0
\end{vmatrix}
-
\begin{vmatrix}
\tilde{\xi}_1/a & \tilde{\xi}_2/b & \tilde{\xi}_3/b \\
0 & - \xi_3'/\sqrt{b} & \xi_2'/\sqrt{b} \\
-b \tilde{\xi}_1^2 & - a \tilde{\xi}_1 \tilde{\xi}_2 & - a \tilde{\xi}_1 \tilde{\xi}_3
\end{vmatrix}
\right)
\\
&\sim 
\begin{vmatrix}
1 & 0 & 0 \\
0 & -\xi_3'/\sqrt{b} & \xi_2'/\sqrt{b} \\
\tilde{\xi}_1/a & \tilde{\xi}_2/b & \tilde{\xi}_3/b
\end{vmatrix}
\sim (\xi_3' \tilde{\xi}_3 + \xi_2' \tilde{\xi}_2).
\end{split}
\end{equation*}
\end{proof}
In the following we give explicit formulae for the resolvents and for limiting operators in two dimensions:
\begin{proposition}
\label{prop:Explicit2d}
Let $d=2$ and
\begin{equation*}
M^2(A,B) = 
\begin{pmatrix}
\frac{A+B}{2 \mu} ((\xi_2')^2 \varepsilon_{11} - (\xi_1' \xi_2') \varepsilon_{12}) & \frac{A+B}{2 \mu}((\xi_2')^2 \varepsilon_{21} - \xi_1' \xi_2' \varepsilon_{22}) & \frac{\xi_2'}{2 \mu}(A-B) \\
\frac{A+B}{2 \mu}((\xi_1')^2 \varepsilon_{21} - \xi_1' \xi_2' \varepsilon_{11}) & \frac{A+B}{2 \mu}((\xi_1')^2 \varepsilon_{22} - \varepsilon_{12} (\xi_1') (\xi_2')) & \frac{\xi_1'}{2 \mu}(B-A) \\
\frac{A-B}{2}(\xi_2' \varepsilon_{11} - \xi_1' \varepsilon_{21}) & \frac{B-A}{2} (\xi_1' \varepsilon_{22} - \xi_2' \varepsilon_{21}) & \frac{A+B}{2}
\end{pmatrix},
\end{equation*}
furthermore,
\begin{equation*}
M^2_{c} = \frac{1}{i \omega \mu}
\begin{pmatrix}
\varepsilon_{22} (\xi_1')^2 - \varepsilon_{12} \xi_1' \xi_2' & \varepsilon_{22} \xi_1' \xi_2' - \varepsilon_{12} (\xi_2')^2 & 0 \\
\varepsilon_{11} \xi_1' \xi_2' - \varepsilon_{12} (\xi_1')^2 & \varepsilon_{11} (\xi_2')^2 - \varepsilon_{12} \xi_1' \xi_2' & 0 \\
0 & 0 & 0
\end{pmatrix}
.
\end{equation*}
Then, we have for $\omega \in \C \backslash \R$ and almost all $\xi \in \R^2$:
\begin{equation*}
(P(\omega,D)^{-1} u) \widehat (\xi) = (M^2(A,B) + M^2_c) \hat{u}(\xi)
\end{equation*}
with 
\begin{equation*}
A= \frac{1}{i(\omega - \| \xi \|_{\varepsilon'})}, \quad B = \frac{1}{i(\omega + \| \xi \|_{\varepsilon'})}.
\end{equation*}

\medskip

For $\omega > 0$, $\beta \in C^\infty_c(\R^2)$, and $u \in \mathcal{S}(\R^2)^3$, we find
\begin{equation*}
P(\omega \pm i \delta,D)^{-1} \beta(D) u \to P^{loc}_{\pm}(\omega) \beta(D) u
\end{equation*}
with
\begin{equation*}
P_{\pm}^{loc}(\omega) \beta(D) u (x) = \frac{1}{(2 \pi)^2} \int_{\R^2} e^{ix.\xi} (M^2(A,B)+M^2_c)
\beta(\xi) \hat{u}(\xi), 
\end{equation*}
where
\begin{equation*}
A= \frac{1}{i} \{ v.p. \frac{1}{\omega - \| \xi \|_{\varepsilon'}} \mp i \pi \delta(\omega - \| \xi \|_{\varepsilon'}) \}, \quad B = \frac{1}{i(\omega + \| \xi \|_{\varepsilon'})}.
\end{equation*}
\end{proposition}
\begin{proof}
The first claim follows from computing $p^{-1}(\omega,\xi)$ (cf. Lemma \ref{lem:Diagonalization2d}). We decompose
\begin{equation*}
\begin{split}
m^{-1}(\xi) &= m_1(\xi) + m_2(\xi) \\
&=
\begin{pmatrix}
0 & 0 & 0 \\
\frac{\xi_1' \varepsilon_{21} - \xi_2' \varepsilon_{11}}{2} & \frac{\varepsilon_{22} \xi_1' - \varepsilon_{21} \xi_2'}{2} & - \frac{1}{2} \\
\frac{\xi_2' \varepsilon_{11} - \xi_1' \varepsilon_{12}}{2} & \frac{\xi_2' \varepsilon_{12} - \xi_1' \varepsilon_{22}}{2} & - \frac{1}{2}
\end{pmatrix}
+
\begin{pmatrix}
\mu^{-1} \xi_1' & \mu^{-1} \xi_2' & 0 \\
0 & 0 & 0 \\
0 & 0 & 0
\end{pmatrix}
\end{split}
\end{equation*}
based on the observation that $m_2(\xi) v(\xi) = 0$ for $\xi_1 v_1(\xi) + \xi_2 v_2(\xi) = 0$. We compute for $A$ and $B$ as in the first claim:
\begin{equation*}
M^2(A,B) = m(\xi) d(\omega,\xi)^{-1} m_1(\xi).
\end{equation*}
The computation is simplified by noting that:
\begin{equation*}
\begin{pmatrix}
0 & 0 & 0 \\
0 & A & 0 \\
0 & 0 & B
\end{pmatrix}
\begin{pmatrix}
0 & 0 & 0 \\
\frac{\xi_1' \varepsilon_{21} - \xi_2' \varepsilon_{11}}{2} & \frac{\varepsilon_{22} \xi_1' - \varepsilon_{21} \xi_2'}{2} & - \frac{1}{2} \\
\frac{\xi_2' \varepsilon_{11} - \varepsilon_{12} \xi_1'}{2} & \frac{\xi_2' \varepsilon_{12} - \xi_1' \varepsilon_{22}}{2} & - \frac{1}{2}
\end{pmatrix}
=
\begin{pmatrix}
0 & 0 & 0 \\
\frac{ A( \xi_1' \varepsilon_{21} - \xi_2' \varepsilon_{11})}{2} & \frac{A (\varepsilon_{22} \xi_1' - \varepsilon_{21} \xi_2'}{2} & - \frac{A}{2} \\
\frac{B(\xi_2' \varepsilon_{11} - \varepsilon_{12} \xi_1')}{2} & \frac{B(\xi_2' \varepsilon_{12} - \xi_1' \varepsilon_{22})}{2} & - \frac{B}{2}
\end{pmatrix}
.
\end{equation*}
We find for $m(\xi) d(\omega,\xi)^{-1} m_1(\xi)$:
\begin{equation*}
\begin{pmatrix}
\frac{A+B}{2 \mu} ((\xi_2')^2 \varepsilon_{11} - (\xi_1' \xi_2') \varepsilon_{12}) & \frac{A+B}{2 \mu}((\xi_2')^2 \varepsilon_{21} - \xi_1' \xi_2' \varepsilon_{22}) & \frac{\xi_2'}{2 \mu}(A-B) \\
\frac{A+B}{2 \mu}((\xi_1')^2 \varepsilon_{21} - \xi_1' \xi_2' \varepsilon_{11}) & \frac{A+B}{2 \mu}((\xi_1')^2 \varepsilon_{22} - \varepsilon_{12} (\xi_1') (\xi_2')) & \frac{\xi_1'}{2 \mu}(B-A) \\
\frac{A-B}{2}(\xi_2' \varepsilon_{11} - \xi_1' \varepsilon_{21}) & \frac{B-A}{2} (\xi_1' \varepsilon_{22} - \xi_2' \varepsilon_{21}) & \frac{A+B}{2}
\end{pmatrix}
\end{equation*}
In $M^2_c = m(\xi) d(\omega,\xi)^{-1} m_2(\xi)$ we have separated the contribution of non-trivial charges. The second claim follows with the same computation from Sokhotsky's formula.
\end{proof}

For $d=3$, we define $M^3(A,B,C,D) \in \C^{6 \times 6}$:
\begin{align*}
M^3_{11} &= \frac{a(C+D)(\tilde{\xi}_2^2 + \tilde{\xi}_3^2)}{2}, \quad M^3_{12} = -\frac{b(C+D)\tilde{\xi}_1 \tilde{\xi}_2}{2}  , M^3_{13} = - \frac{b(C+D)\tilde{\xi}_1 \tilde{\xi}_3 }{2}, \\
 M^3_{14} &= 0, \quad M^3_{15} = \frac{(D-C)\tilde{\xi}_3}{2}, \quad M^3_{16} = \frac{(C-D)\tilde{\xi}_2}{2}.
\end{align*}
Furthermore,
\begin{align*}
M^3_{21} &= - \frac{a(C+D)\tilde{\xi}_1 \tilde{\xi}_2}{2} , \quad M^3_{22} = \frac{(A+B)\xi_3^2}{2(\xi_2^2 + \xi_3^2)} + \frac{b(C + D)\tilde{\xi}_1^2 \xi_2^2}{2(\xi_2^2+ \xi_3^2)} , \\
M^3_{23} &= - \frac{(A+B)\xi_2 \xi_3}{2(\xi_2^2 + \xi_3^2)} + \frac{b(C+D)\tilde{\xi}^2_1 \xi_2 \xi_3}{2(\xi_2^2 + \xi_3^2)}, \quad M^3_{24} = \frac{\xi_3' (A+B)}{2 \sqrt{b}}, \\
M^3_{25} &= \frac{(B-A)\xi_1' \xi_2 \xi_3}{2 \sqrt{b} (\xi_2^2 + \xi_3^2)}  + \frac{(C-D)\tilde{\xi}_1 \xi_2 \xi_3 }{2(\xi_2^2 + \xi_3^2)}, \\
M^3_{26} &= \frac{(B-A) \xi_1' \xi_3^2}{2 \sqrt{b} (\xi_2^2 + \xi_3^2)} + \frac{(D-C) \tilde{\xi}_1 \xi_2^2}{2(\xi_2^2+ \xi_3^2)}.
\end{align*}
Next,
\begin{align*}
M^3_{31} &= -\frac{a(C+D) \tilde{\xi}_1 \tilde{\xi}_3}{2}, \quad M^3_{32} = -\frac{(A+B) \xi_2 \xi_3}{2(\xi_2^2 + \xi_3^2)} + \frac{b(C+D)\tilde{\xi}_1^2 \xi_2 \xi_3}{2(\xi_2^2 + \xi_3^2)}, \\
M^3_{33} &= \frac{(A+B)\xi_2^2 }{2(\xi_2^2 + \xi_3^2)} + \frac{b(C+D)\tilde{\xi}_1^2 \xi_3^2}{2(\xi_2^2 + \xi_3^2)}, \quad M^3_{34} = \frac{(B-A) \xi_2'}{2 \sqrt{b}} , \\
M^3_{35} &= \frac{(A-B)\xi_1' \xi_2^2}{2 \sqrt{b}(\xi_2^2 + \xi_3^2) } + \frac{(C-D)\tilde{\xi}_1 \xi_3^2}{2(\xi_2^2 + \xi_3^2)} , \quad 
M^3_{36} = \frac{(A-B)\xi_1' \xi_2 \xi_3}{2 \sqrt{b} (\xi_2^2 + \xi_3^2)} + \frac{(D-C)\tilde{\xi}_1 \xi_2 \xi_3}{2(\xi_2^2+ \xi_3^2)}. 
\end{align*}
\begin{align*}
M^3_{41} &= 0, \quad M^3_{42} = \frac{\sqrt{b} (A-B) \xi_3'}{2}, \\
M^3_{43} &= \frac{\sqrt{b} (B-A)\xi_2'}{2}, \quad M^3_{44} = \frac{(A+B) (\xi_2'^2 + \xi_3'^2)}{2} , \\
M^3_{45} &= - \frac{(A+B) \xi_1' \xi_2'}{2} , \quad M^3_{46} = - \frac{(A+B) \xi_1' \xi_3'}{2}.
\end{align*}
\begin{align*}
M^3_{51} &= \frac{a(D-C)}{2} \tilde{\xi}_3, \quad M^3_{52} = \frac{\sqrt{b} (B-A) \xi_1' \xi_2 \xi_3}{2(\xi_2^2 + \xi_3^2)} + \frac{b(C-D) \tilde{\xi}_1 \xi_2 \xi_3}{2 (\xi_2^2 + \xi_3^2)} , \\
M^3_{53} &= \frac{\sqrt{b} (A-B) \xi_1' \xi_2^2}{2 (\xi_2^2 + \xi_3^2)} + \frac{b(C-D) \tilde{\xi}_1 \xi_3^2}{2(\xi_2^2 + \xi_3^2)}, \quad M^3_{54} = -\frac{(A+B) \xi_1' \xi_2'}{2} , \\
M^3_{55} &= \frac{(A+B) \xi_1'^2 \xi_2^2}{2 (\xi_2^2+ \xi_3^2)}  + \frac{(C+D) \xi_3^2}{2( \xi_2^2 +  \xi_3^2)}, \quad M^3_{56} = \frac{(A+B) \xi_1'^2 \xi_2 \xi_3}{2(\xi_2^2 + \xi_3^2)}  - \frac{(C+D) \xi_2 \xi_3}{2(\xi_2^2+ \xi_3^2)} .
\end{align*}
Lastly,
\begin{align*}
M^3_{61} &= \frac{a(C-D) \tilde{\xi}_2 }{2}, \quad M^3_{62} = \frac{\sqrt{b} ( B-A)\xi_1' \xi_3^2 }{2(\xi_2^2 + \xi_3^2)}  + \frac{b(D-C) \tilde{\xi}_1 \xi_2^2}{2(\xi_2^2 + \xi_3^2)}, \\
M^3_{63} &= \frac{\sqrt{b}( A-B) \xi_1' \xi_2 \xi_3}{2(\xi_2^2 + \xi_3^2)} + \frac{b(D-C) \tilde{\xi}_1 \xi_2 \xi_3}{2(\xi_2^2 + \xi_3^2)} , \quad M^3_{64} = - \frac{(A+B) \xi_1' \xi_3' }{2}, \\
M^3_{65} &= \frac{(A+B) \xi_1'^2 \xi_2 \xi_3}{2 (\xi_2^2 + \xi_3^2)} - \frac{(C+D) \xi_2 \xi_3}{2(\xi_2^2 + \xi_3^2)} , \quad M^3_{66} = \frac{(A+B) \xi_1'^2 \xi_3^2}{2 ( \xi_2^2 + \xi_3^2)}  + \frac{(C+D) \xi_2^2}{2( \xi_2^2 + \xi_3^2)}.
\end{align*}
We let moreover
\begin{equation*}
M^3_c = \frac{1}{i \omega}
\begin{pmatrix}
b \tilde{\xi}_1^2  & b \tilde{\xi}_1 \tilde{\xi}_2  & b \tilde{\xi}_1 \tilde{\xi}_3 & 0 & 0 & 0 \\
a \tilde{\xi}_1 \tilde{\xi}_2  & a \tilde{\xi}_2^2  & a \tilde{\xi}_2 \tilde{\xi}_3  & 0 & 0 & 0 \\
a \tilde{\xi}_1 \tilde{\xi}_3  & a \tilde{\xi}_2 \tilde{\xi}_3  & a \tilde{\xi}_3^2  & 0 & 0 & 0 \\
0 & 0 & 0 & \xi_1'^2 & \xi_1' \xi'_2 & \xi_1' \xi'_3 \\
0 & 0 & 0 & \xi'_1 \xi_2' & \xi_2'^2 & \xi_2' \xi'_3 \\
0 & 0 & 0 & \xi'_1 \xi_3' & \xi'_2 \xi_3' & \xi_3'^2 
\end{pmatrix}
.
\end{equation*}

We have the following analog of Proposition \ref{prop:Explicit2d}:
\begin{proposition}
\label{prop:Explicit3d}
Let $d=3$. We find for $\omega \in \C \backslash \R$ and almost all $\xi \in \R^3$
\begin{equation*}
(P(\omega,D)^{-1} u) \widehat (\xi) = (M^3(A,B,C,D) + M^3_c) \hat{u}(\xi)
\end{equation*}
with 
\begin{equation*}
A= \frac{1}{i(\omega - \sqrt{b} \| \xi \|)}, \; B = \frac{1}{i(\omega + \sqrt{b} \| \xi \|)}, \; C= \frac{1}{i(\omega - \| \xi \|_\varepsilon)}, \; D= \frac{1}{i(\omega + \| \xi \|_\varepsilon)}.
\end{equation*}

\medskip

For $\omega > 0$, $\beta \in C^\infty_c(\R^3)$, and $u \in \mathcal{S}(\R^3)^6$, we find
\begin{equation*}
P(\omega \pm i \delta,D)^{-1} \beta(D) u \to P^{loc}_{\pm}(\omega) \beta(D) u \text{ in } (\mathcal{S}'(\R^3))^6
\end{equation*}
with
\begin{equation*}
P_{\pm}^{loc}(\omega) \beta(D) u (x) = \frac{1}{(2 \pi)^3} \int_{\R^3} e^{ix.\xi} ( M^3(A,B,C,D) + M^3_c)
\beta(\xi) \hat{u}(\xi), 
\end{equation*}
where
\begin{equation*}
\begin{split}
A &= \frac{1}{i} \{ v.p. \frac{1}{\omega - \sqrt{b} \| \xi \|} \mp i \pi \delta(\omega - \sqrt{b} \| \xi \|) \}, \; B = \frac{1}{i(\omega + \sqrt{b} \| \xi \|)}, \\
C &= \frac{1}{i} \{ v.p. \frac{1}{\omega -  \| \xi \|_{\varepsilon} } \mp i \pi \delta(\omega -  \| \xi \|_{\varepsilon}) \}, \quad \quad D = \frac{1}{i(\omega +  \| \xi \|_{\varepsilon})}.
\end{split}
\end{equation*}
\end{proposition}

\section*{Acknowledgements}

Funded by the Deutsche Forschungsgemeinschaft (DFG, German Research Foundation)  Project-ID 258734477 – SFB 1173. I would like to thank Lucrezia Cossetti and Rainer Mandel for helpful discussions about the results and context. Moreover, I am much obliged to the anonymous referees whose insightful comments clearly improved the presentation.

\end{document}